\documentclass[12pt,twoside,reqno]{amsart}
\usepackage{amsmath}
\usepackage{amsfonts}
\usepackage{amssymb}
\usepackage{color}
\usepackage{mathrsfs}
\usepackage{cite}
\newtheorem{theorem}{Theorem}[section]
\newtheorem{corollary}[theorem]{Corollary}
\newtheorem{lemma}[theorem]{Lemma}
\newtheorem{proposition}[theorem]{Proposition}

\newtheorem{remark}[theorem]{Remark}
\numberwithin{equation}{section}
\begin{document}
\title{Mass-conserving self-similar solutions to coagulation-fragmentation equations} 
%\thanks{}
\author{Philippe Lauren\c{c}ot}
\address{Institut de Math\'ematiques de Toulouse, UMR~5219, Universit\'e de Toulouse, CNRS \\ F--31062 Toulouse Cedex 9, France}
\email{laurenco@math.univ-toulouse.fr}


\date{\today}

%%%%%%%%%%%%%%%%
%%%%%%%%%%%%%%%%
\begin{abstract}
Existence of mass-conserving self-similar solutions with a sufficiently small total mass is proved for a specific class of homogeneous coagulation and fragmentation coefficients. The proof combines a dynamical approach to construct such solutions for a regularised coagulation-fragmentation equation in scaling variables and a compactness method.
\end{abstract}
%%%%%%%%%%%%%%%%
%%%%%%%%%%%%%%%%

\maketitle

%
%     HEADLINES
%
\pagestyle{myheadings}
\markboth{\sc{Philippe Lauren\c cot}}{\sc{Mass-conserving self-similar solutions to C-F equations}}

%%%%%%%%%%%%%%%%
%%%%%%%%%%%%%%%%
\section{Introduction}\label{sec1}
%%%%%%%%%%%%%%%%
%%%%%%%%%%%%%%%%

Coagulation-fragmentation equations are mean-field models describing the time evolution of the size distribution function $f$ of a system of particles varying their sizes due to the combined effect of binary coalescence and multiple breakage. The dynamics of the size distribution function $f(t,x)$ of particles of size $x\in (0,\infty)$ at time $t>0$ is governed by the nonlinear integral equation
\begin{subequations}\label{a1}
\begin{align}
\partial_t f(t,x) & = \mathcal{C}f(t,x) + \mathcal{F}f(t,x)\ , \qquad (t,x) \in (0,\infty)^2\ , \label{a1a} \\
f(0,x) & = f^{in}(x)\ , \qquad x\in (0,\infty)\ , \label{a1d}
\end{align}
where
\begin{equation}
\mathcal{C}f(x) := \frac{1}{2} \int_0^x K(y,x-y) f(x-y) f(y)\ \mathrm{d}y - \int_0^\infty K(x,y) f(x) f(y)\ \mathrm{d}y\ , \qquad x\in (0,\infty)\ , \label{a1b}
\end{equation}
and 
\begin{equation}
\mathcal{F}f(x) := - a(x) f(x) + \int_x^\infty a(y) b(x,y) f(y)\ \mathrm{d}y\ , \qquad x\in (0,\infty)\ , \label{a1c}
\end{equation}
\end{subequations}
account for the coagulation and fragmentation processes, respectively. In \eqref{a1b}, the coagulation kernel $K$ is a non-negative and symmetric function defined on $(0,\infty)^2$ and $K(x,y)=K(y,x)$ is the rate at which two particles of respective sizes $x$ and $y$ collide and merge. In \eqref{a1c}, $a(x)$ is the overall fragmentation rate of particles of size $x$ and the distribution of the sizes of fragments resulting from the splitting of a particle of size $y$ is the daughter distribution function $x\mapsto b(x,y)$. Since we discard the possibility of loss of matter during breakup, $b$ is assumed to satisfy
\begin{equation}
\int_0^y x b(x,y)\ \mathrm{d}x = y\ , \qquad y>0\ , \;\;\text{ and }\;\; b(x,y)= 0 \ , \qquad x>y>0\ ; \label{a2}
\end{equation}
that is, the fragmentation of a particle of size $y$ only produces particles of smaller sizes and no matter is lost. Coagulation being also a mass-conserving process, we expect that matter is conserved throughout time evolution; that is,
\begin{equation}
M_1(f(t)) := \int_0^\infty x f(t,x)\ \mathrm{d}x = \varrho = M_1(f^{in}) := \int_0^\infty x f^{in}(x)\ \mathrm{d}x\ , \qquad t\ge 0\ . \label{a3}
\end{equation}
Breakdown in finite time of the identity \eqref{a3} may actually occur; that is, there is $T_l\in [0,\infty)$ such that
\begin{equation*}
M_1(f(t)) < M_1(f^{in})\ , \qquad t>T_l\ . 
\end{equation*} 
This feature is due, either to a runaway growth generated by a coagulation kernel increasing rapidly for large sizes, a phenomenon known as \textit{gelation} \cite{Leyv83, Leyv03, LeTs81}, or to the appearance of dust resulting from an overall fragmentation rate $a$ which is unbounded as $x\to 0$, a phenomenon referred to as \textit{shattering} \cite{Fili61, McZi87}. Loosely speaking, for the coagulation and fragmentation coefficients given by
\begin{subequations}\label{a4}
\begin{equation}
K(x,y) = K_0 \left( x^\alpha y^{\lambda-\alpha} + x^{\lambda-\alpha} y^\alpha \right)\ , \qquad (x,y)\in (0,\infty)^2\ , \label{a4a}
\end{equation}
with $\alpha\in [0,1]$, $\lambda\in [2\alpha,1+\alpha]$, and $K_0 > 0$, and
\begin{equation}
a(x) = a_0 x^\gamma\ , \quad b(x,y) = b_\nu(x,y) := (\nu+2) x^\nu y^{-\nu-1}\ , \qquad 0<x<y\ , \label{a4b}
\end{equation}
\end{subequations}
with $\gamma\in\mathbb{R}$, $\nu\in (-2,\infty)$, and $a_0> 0$, gelation after a finite time occurs when $\alpha>1/2$ in \eqref{a4a} and $\gamma\in (0,\lambda-1)$ in \eqref{a4b} \cite{EMP02, ELMP03, Jeon98, Laur00, Leyv83, LeTs81}, while shattering is observed when $\gamma<0$ in \eqref{a4b} and there is no coagulation $(K_0=0)$ \cite{Bana06, Fili61, McZi87}. In contrast, mass-conserving solutions to \eqref{a1} satisfying \eqref{a3} for all $t\ge 0$ exist when, either $\lambda\in [0,1]$ and $\gamma\ge 0$, or $\lambda\in (1,2]$ and $\gamma>\lambda-1$ \cite{BaCa90, BaLa11, BaLa12a, BLL13, BLLxx, daCo95, ELMP03, EMRR05, Laur18, LaMi02b, Stew89, Stew91, Whit80}. The previous discussion reveals that the value $\gamma=\lambda-1>0$ is a borderline case with respect to the occurrence of the gelation phenomenon. Indeed, on the one hand, when $\lambda\in (1,2]$, $\gamma=\lambda-1$, and $\alpha>-\nu-1$ in \eqref{a4}, mass-conserving solutions to \eqref{a1} on $[0,\infty)$ exist when $M_1(f^{in})$ is sufficiently small \cite{LaurXX}, which is in accordance with numerical simulations performed in \cite{Pisk12b} for the particular choice 
\begin{equation}
\alpha=1\ , \qquad \lambda=2\ , \qquad \gamma=1\ , \qquad \nu=0\ . \label{ex1}
\end{equation}
On the other hand, gelation (in finite time) takes place when $\alpha=1$, $\lambda=2$, $\gamma=1$, $\nu>-1$,  and $M_1(f^{in})$ is large enough \cite{BLLxx, Pisk12b, ViZi89}.

Besides, the choice $\gamma=\lambda-1>0$ in \eqref{a4} has another interesting feature. Indeed, in this case, equation~\eqref{a1a} satisfies a scale invariance which complies with the conservation of matter \eqref{a3}. More precisely, if $f$ is a solution to \eqref{a1a} and $r>0$, then the function $f_r$ defined by
\begin{equation}
f_r(t,x) := r^2 f(r^{1-\lambda}t,r x)\ , \qquad (t,x)\in [0,\infty) \times (0,\infty)\ , \label{a5}
\end{equation}
is also a solution to \eqref{a1a} and $M_1(f_r(t))=M_1(f(r^{1-\lambda}t))$ for $t\ge 0$. We then look for particular solutions to \eqref{a1a} which are left invariant by the transformation \eqref{a5}, that is, $f_r=f$ for all $r>0$; that is, according to \eqref{a5}, $r^2 f(r^{1-\lambda}t,rx) = f(t,x)$ for all $(r,t,x)\in (0,\infty)^3$. The choice $r=t^{1/(\lambda-1)}$ in the previous identity gives 
\begin{equation*}
f(t,x) = t^{2/(\lambda-1)} f\left( 1,x t^{1/(\lambda-1)} \right)\ , \qquad (t,x)\in (0,\infty)^2\ ,
\end{equation*} 
and raises the question of the existence of mass-conserving self-similar solutions of the form
\begin{equation}
(t,x) \longmapsto t^{2/(\lambda-1)} \psi\left( x t^{1/(\lambda-1)} \right)\ , \qquad (t,x)\in (0,\infty)^2\ . \label{a6}
\end{equation}
In \eqref{a6}, the profile $\psi$ is yet to be determined and is requested to have a finite total mass $M_1(\psi)=\varrho\in (0,\infty)$. According to the numerical simulations performed in \cite{Pisk12b}, such solutions exist for sufficiently small values of $\varrho$ and are expected to describe the long term dynamics of mass-conserving solutions to \eqref{a1} with the same total mass $\varrho$. Thus,  the existence, uniqueness, and properties of mass-conserving self-similar solutions to \eqref{a1a} of the form \eqref{a6} are of high interest. 

The purpose of this paper is to provide one step in that direction and figure out whether self-similar solutions to \eqref{a1a} of the form \eqref{a5} do exist when $\gamma=\lambda-1>0$ in \eqref{a4}. Such a quest is not hopeless. Indeed, on the one hand, when the parameters in \eqref{a4} are given by \eqref{ex1}, their existence is supported by numerical simulations performed in \cite{Pisk12b}, which indicate that there exist mass-conserving self-similar solutions to \eqref{a1a} of the form \eqref{a6} with $M_1(\psi)=\varrho$, provided the ratio $a_0/(\varrho K_0)$ is large enough. On the other hand, if
\begin{equation}
\alpha=1\ , \qquad \lambda=2\ , \qquad \gamma=1\ , \qquad \nu=-1\ , \label{ex2}
\end{equation}
then, for any $\varrho>0$, the existence of a unique mass-conserving self-similar solution to \eqref{a1a} of the form \eqref{a6} with $M_1(\psi)=\varrho$ is shown in \cite{LavR15} and this particular solution is a global attractor for the dynamics of \eqref{a1} when the initial condition $f^{in}$ satisfies $M_1(f^{in})=\varrho$. The approach developed in \cite{LavR15} heavily relies on the specific structure of \eqref{a1a} for the choice of parameters \eqref{ex2}, which allows us to use the Laplace transform, and is thus not likely to be adapted to the more general setting considered herein. Instead, we first construct mass-conserving self-similar solution to \eqref{a1a} of the form \eqref{a6} for a restricted class of daughter distribution functions $b$ by a dynamical approach and carefully keep track of the dependence of the estimates on the various parameters involved in $K$, $a$, and $b$. We next use a compactness method to extend the existence result to a broader class of $b$. 

Specifically, we consider
\begin{subequations}\label{aaCFC}
\begin{equation}
\lambda \in (1,2]\ , \quad \gamma:= \lambda-1 \in (0,1]\ , \quad \alpha\in \left[ \max\left\{ \frac{1}{2} , \lambda-1 \right\} , \frac{\lambda}{2} \right]\ , \label{aa1}
\end{equation}
and assume that the overall fragmentation rate $a$ and the coagulation kernel $K$ are given by 
\begin{eqnarray}
a(x) & = & a_0 x^{\lambda-1}\ , \qquad x\in (0,\infty)\ , \label{aa2} \\
K(x,y) & = & K_0 \left( x^\alpha y^{\lambda-\alpha} + x^{\lambda-\alpha} y^\alpha \right)\ , \qquad (x,y)\in (0,\infty)^2\ , \label{aa3}
\end{eqnarray}
for some positive constants $a_0$ and $K_0$. We assume further that the daughter distribution function $b$ has the scaling form
\begin{equation}
b(x,y) = \frac{1}{y} B\left( \frac{x}{y} \right)\ , \qquad 0<x<y\ , \label{aa4}
\end{equation}
where
\begin{equation}
B\ge 0 \;\text{ a.e. in }\; (0,1)\ , \quad B\in L^1((0,1),z\mathrm{d}z)\ , \quad \int_0^1 zB(z)\ \mathrm{d}z =1\ , \label{aa5}
\end{equation}
and there is $\nu\in (-2,0]$ such that
\begin{equation}
\mathfrak{b}_{m,p} := \int_0^1 z^m B(z)^p\ \mathrm{d}z < \infty   \label{aa6}
\end{equation}
for all $(m,p)\in\mathcal{A}_\nu$, the set $\mathcal{A}_\nu$ being defined by
\begin{equation}
\mathcal{A}_\nu :=\{ (m,p)\in (-1,\infty)\times [1,\infty)\ :\ m+p\nu>-1\}\ . \label{aa7}
\end{equation}
\end{subequations}
Observe that $\mathcal{A}_\nu$ is non-empty since 
\begin{subequations}\label{aa12}
\begin{equation}
(m,1)\in \mathcal{A}_\nu \;\;\text{ for all }\;\; m>-\nu-1\ . \label{aa12a}
\end{equation}
Also, if $(m,1)\in \mathcal{A}_\nu$, then
\begin{equation}
(m,p)\in \mathcal{A}_\nu \;\;\text{ for all }\;\; p\in \left[ 1 , \frac{m+1}{|\nu|} \right)\ . \label{aa12b}
\end{equation}
\end{subequations}
We finally assume that the small size behaviour of the coagulation kernel $K$ is related to the possible singularity of $B$ for small sizes and require
\begin{equation}
-\nu-1 < \alpha\ . \label{aa8}
\end{equation}
Since $(-\nu/2,1]\in \mathcal{A}_\nu$ by \eqref{aa12}, we infer from \eqref{aa6} and the inequality
\begin{equation*}
\int_0^1 z |\ln z| B(z)\ \mathrm{d}z \le \sup_{z\in (0,1)}\left\{ z^{(2+\nu)/2} |\ln{z}| \right\} \int_0^1 z^{-\nu/2} B(z)\ \mathrm{d}z = \frac{2\mathfrak{b}_{-\nu/2,1}}{e(\nu+2)}\ ,
\end{equation*} 
that
\begin{equation}
\mathfrak{b}_{\ln} := \int_0^1 z |\ln z| B(z)\ \mathrm{d}z < \infty\ . \label{aa9}
\end{equation}
We then set
\begin{equation}
\varrho_\star := \frac{a_0 \mathfrak{b}_{\ln}}{2 K_0 \ln{2}}\ . \label{aa10}
\end{equation}

For $m\in\mathbb{R}$, we define the weighted $L^1$-space $X_m$ and the moment $M_m(h)$ of order $m$ of $h\in X_m$ by
\begin{equation*}
X_m := L^1((0,\infty),x^m\mathrm{d}x)\ , \qquad M_m(h) := \int_0^\infty x^m h(x)\ \mathrm{d}x\ .
\end{equation*}
We also denote the positive cone of $X_m$ by $X_m^+$, while $X_{m,w}$ denotes the space $X_m$ endowed with its weak topology.

For the above described class of coagulation and fragmentation coefficients, the main result of this paper guarantees the existence of at least one mass-conserving self-similar solution to \eqref{a1a} of the form \eqref{a6} (up to a rescaling, see Remark~\ref{rema2} below) with a sufficiently small total mass $\varrho$.

%%%%%%%%%%%%%%%%
\begin{theorem}\label{ThmP2}
Consider coagulation and fragmentation coefficients $K$, $a$, and $b$ satisfying \eqref{aaCFC} and fix two auxiliary parameters
\begin{equation}
m_0 \in (-\nu-1,\alpha) \cap [0,1)\ , \qquad m_1 := \max\{m_0,2-\lambda\}\in (0,1)\ . \label{aa11}
\end{equation}
Let $\varrho\in (0,\varrho_\star)$.
\begin{itemize}
\item[(a)] There are $q_1\in (1,2)$ (defined in \eqref{p1} below) and a non-negative profile 
\begin{equation}
\varphi\in X_1^+\cap L^{q_1}((0,\infty),x^{m_1}\mathrm{d}x) \cap \bigcap_{m\ge m_0} X_m\ , \qquad M_1(\varphi) = \varrho\ , \label{a6a}
\end{equation}
such that $(m_1,q_1)\in\mathcal{A}_\nu$ and
\begin{align}
\int_0^\infty \left[ \vartheta(x) - x \partial_x\vartheta(x) \right] \varphi(x)\ \mathrm{d}x & = \frac{1}{2} \int_0^\infty \int_0^\infty K(x,y) \chi_\vartheta(x,y) \varphi(x) \varphi(y)\ \mathrm{d}y\mathrm{d}x \nonumber \\
& \qquad - \int_0^\infty a(y) N_{\vartheta}(y) \varphi(y)\ \mathrm{d}y \label{x8}
\end{align}
for all $\vartheta\in \Theta_1$, where
\begin{equation}
\Theta_1 := \left\{ \vartheta\in W^{1,\infty}(0,\infty)\ :\ \vartheta(0)=0 \right\} \ , \label{a6c}
\end{equation}
and
\begin{align}
\chi_\vartheta(x,y) & := \vartheta(x+y) -\vartheta(x) - \vartheta(y)\ , \qquad (x,y)\in (0,\infty)^2\ , \label{bchi}\\
N_\vartheta(y) & := \vartheta(y) - \int_0^y \vartheta(x) b(x,y)\ \mathrm{d}x\ , \qquad y\in (0,\infty)\ . \label{bN}
\end{align}
\item[(b)] The function $F_S$ defined by
\begin{equation}
F_S(t,x) := s_\lambda(t)^2 \varphi\left( x s_\lambda(t) \right)\ , \qquad (t,x)\in [0,\infty)\times (0,\infty)\ , \label{a6b}
\end{equation}
with $s_\lambda(t) := (1+(\lambda-1)t)^{1/(\lambda-1)}$, $t\ge 0$, is a mass-conserving weak solution to \eqref{a1} on $[0,\infty)$ with initial condition $f^{in}=\varphi$ in the following sense: for any $T>0$,
\begin{equation*}
F_S\in C([0,T],X_{m_1,w}) \cap C([0,T],X_{1,w}) \cap L^\infty((0,T),X_{m_0}) 
\end{equation*} 
and satisfies
\begin{align}
\int_0^\infty (F_S(t,x)-\varphi(x)) \vartheta(x)\ \mathrm{d}x & = \frac{1}{2} \int_0^t \int_0^\infty \int_0^\infty K(x,y) \chi_\vartheta(x,y) F_S(s,x) F_S(s,y)\ \mathrm{d}y\mathrm{d}x\mathrm{d}s \nonumber\\
& \qquad - \int_0^t \int_0^\infty a(x) N_\vartheta(x) F_S(s,x)\ \mathrm{d}x\mathrm{d}s\ , \label{b0}
\end{align}
 for all $t\in (0,\infty)$ and $\vartheta\in\Theta_{m_1}$, where $\Theta_0:= L^\infty(0,\infty)$ and
\begin{equation*}
\Theta_m := \left\{ \vartheta\in C^m([0,\infty))\cap L^\infty(0,\infty)\ :\ \vartheta(0) = 0 \right\}\ , \qquad m\in (0,1)\ .
\end{equation*} 
\end{itemize}
\end{theorem}
%%%%%%%%%%%%%%%%

%%%%%%%%%%%%%%%%
\begin{remark}\label{rema2}
The self-similar ansatz \eqref{a6} differs slightly from that of $F_S$ in Theorem~\ref{ThmP2}, see \eqref{a6b}. However, they can both be mapped to each other, up to an $X_1$-invariant dilation of the profile. Indeed, if $F_S(t,x) = s_\lambda(t)^2 \varphi\left( x s_\lambda(t) \right)$, $(t,x)\in (0,\infty)^2$, is a mass-conserving self-similar solution to \eqref{a1a} of the form \eqref{a6b}, then it is actually well-defined for $(t,x)\in (-1/(\lambda-1),\infty)\times (0,\infty)$. Combining this property with the autonomous character of the coagulation-fragmentation equation \eqref{a1a} implies that $\tilde{F}_S(t,x):= F_S(t-(\lambda-1)^{-1},x)$, $(t,x)\in (0,\infty)^2$, is also a solution to \eqref{a1a} and satisfies
\begin{equation*}
\tilde{F}_S(t,x) :=  t^{2/(\lambda-1)} \psi\left( x t^{1/(\lambda-1)} \right) \ , \qquad (t,x)\in [0,\infty)\times (0,\infty)\ ,
\end{equation*} 
with $\psi(y) = (\lambda-1)^{-2/(\lambda-1)} \varphi\left( y (\lambda-1)^{-1/(\lambda-1)} \right)$, $y>0$. In other words, $\tilde{F}_S$ is a mass-conserving self-similar solution to \eqref{a1a} of the form \eqref{a6} and it has total mass $\varrho$, since $M_1(\varphi)=M_1(\psi)=\varrho$ by \eqref{a6a}.
\end{remark}
%%%%%%%%%%%%%%%%

On the one hand, Theorem~\ref{ThmP2} and Remark~\ref{rema2} provide the existence of mass-conserving self-similar solutions to \eqref{a1} of the form \eqref{a6} with a sufficiently small total mass for the parameters given by \eqref{ex1}, which is in perfect agreement with the numerical simulations performed in \cite{Pisk12b}. It is yet unclear whether $\varrho_\star$ is the largest value of $\varrho$ for which a mass-conserving self-similar solution to \eqref{a1} of the form \eqref{a6} with total mass $\varrho$ exists. However, Theorem~\ref{ThmP2} cannot be valid for any $\varrho>0$ in general. Indeed, when the parameters in \eqref{aaCFC} are given by \eqref{ex1}, gelation occurs for sufficiently large mass, as indicated by explicit computations performed in \cite{Pisk12b, ViZi89} and proved in \cite{BLLxx} when $a_0/(\varrho K_0)<1$. On the other hand, Theorem~\ref{ThmP2} provides the existence of mass-conserving self-similar solutions to \eqref{a1} of the form \eqref{a6} with a sufficiently small total mass for the parameters given by \eqref{ex2}, a result which is far from optimal, since such a  solution exists for any value of the total mass, according to \cite{LavR15}. A possible explanation for this discrepancy is that the absence of a threshold mass is due to the non-integrability as $x\to 0$ of the daughter distribution function $b_{-1}$, which is not really exploited in the proof of Theorem~\ref{ThmP2} below.

Let us now describe the approach we use in this paper to prove Theorem~\ref{ThmP2}. Owing to the homogeneity of $K$, $a$, and $B$, inserting the ansatz \eqref{a6b} in \eqref{a1a} implies that $\varphi$ solves the integro-differential equation
\begin{equation}
y \frac{\mathrm{d}\varphi}{\mathrm{d}y}(y) + 2 \varphi(y) = \mathcal{C}\varphi(y) + \mathcal{F}\varphi(y)\ , \qquad y \in (0,\infty)\ . \label{EqProfile1}
\end{equation}
Unfortunately, the equation~\eqref{EqProfile1} seems hardly tractable as an initial value problem with initial condition at $y=0$. Indeed, on the one hand, the right hand side of \eqref{EqProfile1} depends not only on the past $(0,y)$ of $y$ but also on its future $(y,\infty)$. On the other hand, the left hand side is degenerate, as the factor $y$ in front of $\mathrm{d}\varphi/\mathrm{d}y$ vanishes at $y=0$. Assuming further that $y^2\varphi(y)\to 0$ as $y\to 0$, one can get rid of the derivative in \eqref{EqProfile1} and show that $\varphi$ also satisfies the nonlinear integral equation
\begin{equation}
y^2 \varphi(y) = \int_y^\infty a(x) \varphi(x) \int_0^y x_* b(x_*,x)\ \mathrm{d}x_* \mathrm{d}x - \int_0^y \int_{y-x}^\infty x K(x,x_*) \varphi(x) \varphi(x_*)\ \mathrm{d}x_* \mathrm{d}x \label{EqProfile2}
\end{equation}
for $y\in (0,\infty)$, see \cite{Fili61, vDEr88}. It is however unclear whether this alternative formulation is more helpful than \eqref{EqProfile1} to investigate the existence issue, though it has been extensively used to determine the behaviour for small and large sizes of the profile of mass-conserving self-similar solutions to the coagulation equation \cite{FoLa06a, Leyv03, NiVe11a, NiVe14a, vDEr88}. We thus employ a different approach here, which has already proved successful for the coagulation equation \cite{EMRR05, FoLa05, NiVe14a} and the fragmentation equation \cite{EMRR05, MMP05}. It relies on the construction of a convex and compact subset of $X_1$ which is left invariant by the evolution equation associated to \eqref{EqProfile1}. This evolution equation is actually obtained from \eqref{a1} by using the so-called scaling or self-similar variables. More precisely, recalling that $s_\lambda(t) = (1+(\lambda-1)t)^{1/(\lambda-1)}$, $t\ge 0$, we introduce the scaling variables
\begin{equation*}
s := \ln{s_\lambda(t)}\ , \qquad y := x s_\lambda(t)\ , \qquad (t,x)\in [0,\infty)\times (0,\infty)\ ,
\end{equation*} 
and the rescaled size distribution function 
\begin{equation}
g(s,y) := e^{-2s} f\left( \frac{e^{(\lambda-1)s}-1}{\lambda-1} , y e^{-s} \right)\ , \qquad (s,y)\in  [0,\infty) \times (0,\infty)\ .  \label{a7}
\end{equation}
Equivalently,
\begin{equation}
f(t,x) = s_\lambda(t)^2 g\left( \ln{s_\lambda(t)} , x s_\lambda(t) \right)\ , \qquad (t,x)\in [0,\infty) \times (0,\infty)\ . \label{a8}
\end{equation}
Now, if $f$ is a solution to \eqref{a1}, then $g$ solves
\begin{subequations}\label{a9}
\begin{align}
\partial_s g(s,y) & = - y \partial_y g(s,y) - 2 g(s,y) + \mathcal{C}g(s,y) + \mathcal{F}g(s,y)\ , \qquad (s,y) \in (0,\infty)^2\ , \label{a9a} \\
g(0,y) & = f^{in}(y)\ , \qquad y\in (0,\infty)\ , \label{a9d}
\end{align}
\end{subequations}
Comparing \eqref{EqProfile1} and \eqref{a9a}, we readily see that $\varphi$ is a stationary solution to \eqref{a9a}, so that proving Theorem~\ref{ThmP2} amounts to find a steady-state solution to \eqref{a9a}. To this end, we shall use a consequence of Schauder's fixed point theorem which guarantees the existence of a steady state for a dynamical system defined in a closed subset $Y$ of a Banach space $X$ which leaves invariant a convex and compact subset of $Y$, see \cite[Proposition~22.13]{Aman90} and \cite[Proof of Theorem~5.2]{GPV04} (see also \cite[Theorem~1.2]{EMRR05} for the extension of this result to a Banach space endowed with its weak topology). Applying the just mentioned result requires identifying a suitable functional framework in which, not only \eqref{a9} is well-posed, but also leaves invariant a convex and compact subset of the chosen function space. To achieve this goal, the assumption \eqref{aa6} for any $(m,p)\in \mathcal{A}_\nu$ does not seem to be sufficient and we first construct a family $(b_\varepsilon,B_\varepsilon)_{\varepsilon\in (0,1)}$ of approximations of $(b,B)$, which satisfy not only \eqref{aa4} and \eqref{aa5}, but also \eqref{aa6} for any $(m,p)\in \mathcal{A}_0$ and $B_\varepsilon\in W^{1,1}(0,1)$. We then prove that the corresponding rescaled coagulation-fragmentation equation \eqref{a9} is well-posed in $X_1$ for initial conditions $f^{in}\in X_{m_0}^+\cap X_{1+\lambda}$ satisfying $M_1(f^{in})=\varrho\in (0,\varrho_\star)$. We also show the existence of an invariant convex and compact subset $\mathcal{Z}_\varepsilon$ of $X_1$ for the associated dynamical system. According to the above mentioned result, this analysis guarantees the existence of a stationary solution $\varphi_\varepsilon\in X_1^+$ to \eqref{a9a} satisfying $M_1(\varphi_\varepsilon)=\varrho$. Moreover, it turns out that there is a convex and sequentially weakly compact subset $\mathcal{Z}$ of $X_1$ such that $\mathcal{Z}_\varepsilon\subset \mathcal{Z}$ for all $\varepsilon\in (0,1)$. Consequently, $(\varphi_\varepsilon)_{\varepsilon\in (0,1)}$ is relatively sequentially weakly compact in $X_1$ and the information derived from $\mathcal{Z}$ allows us to prove that cluster points in $X_{1,w}$ of $(\varphi_\varepsilon)_{\varepsilon\in (0,1)}$ as $\varepsilon\to 0$ solve \eqref{EqProfile1}, thereby completing the proof of Theorem~\ref{ThmP2}.

%%%%%%%%%%%%%%%%
\begin{remark}\label{remb}
In the companion paper \cite{LaurXX}, we prove that, given an initial condition $f^{in}\in X_{m_0}^+\cap X_{2\lambda-\alpha}$ satisfying $M_1(f^{in})=\varrho\in (0,\varrho_\star)$, the coagulation-fragmentation equation \eqref{a1} has a unique mass-conserving weak solution on $[0,\infty)$ under the same assumptions \eqref{aaCFC} on the coagulation and fragmentation coefficients. This result is perfectly consistent with the numerical simulations performed in \cite{Pisk12b}, as is Theorem~\ref{ThmP2}.
\end{remark}
%%%%%%%%%%%%%%%%

%%%%%%%%%%%%%%%%
%%%%%%%%%%%%%%%%
\section{Self-similar solutions: a regularised problem} \label{sec2}
%%%%%%%%%%%%%%%%
%%%%%%%%%%%%%%%%

In this section, we assume that $K$, $a$, and $b$ are coagulation and fragmentation coefficients satisfying \eqref{aaCFC} and we fix $\varrho\in (0,\varrho_\star)$. 

As already mentioned, two steps are needed to prove Theorem~\ref{ThmP2} and this section is devoted to the first step; that is,  the proof of Theorem~\ref{ThmP2} for a family $(b_\varepsilon)_{\varepsilon>0}$ of approximations of the daughter distribution function $b$. We begin with the construction of a suitably regularised version of the daughter distribution function $b$. To this end, we fix a non-negative function $\zeta\in C_0^\infty(\mathbb{R})$ such that
\begin{equation*}
\int_{\mathbb{R}} \zeta(z)\ \mathrm{d}z = 1\ , \qquad \mathrm{supp}\; \zeta\subset (-1,1)\ ,
\end{equation*}
and set $\zeta_\varepsilon(z) := \varepsilon^{-2} \zeta(z \varepsilon^{-2})$ for $z\in\mathbb{R}$ and $\varepsilon\in (0,1)$. For $\varepsilon\in (0,1)$, we define
\begin{subequations}\label{be}
\begin{align}
\beta_\varepsilon & := \int_0^1 z \int_\varepsilon^1 \zeta_\varepsilon(z-z_*) B(z_*)\ \mathrm{d}z_*\mathrm{d}z\ , \label{be1}\\
B_\varepsilon(z) & :=  \frac{1}{\beta_\varepsilon} \int_\varepsilon^1 \zeta_\varepsilon(z-z_*) B(z_*)\ \mathrm{d}z_*\ , \qquad z\in (0,1)\ , \label{be2}
\end{align}
and
\begin{equation}
b_\varepsilon(x,y)  := \frac{1}{y} B_\varepsilon\left( \frac{x}{y} \right) \ , \qquad 0<x<y\ . \label{be3}
\end{equation}
\end{subequations}
As we shall see below, see \eqref{app1}, the parameter $\beta_\varepsilon$ is positive for $\varepsilon>0$ sufficiently small, so that $B_\varepsilon$ is well-defined for such values of $\varepsilon$. Indeed, thanks to \eqref{aa5}, \eqref{aa6}, and the properties of $\zeta$, 
\begin{subequations}\label{app}
\begin{equation}
B_\varepsilon\ge 0 \;\;\text{ a.e. in }\;\; (0,1)\ , \quad B_\varepsilon \in L^1((0,1),z\mathrm{d}z)\ , \quad \int_0^1 zB_\varepsilon(z)\ \mathrm{d}z =1\ , \label{app0}
\end{equation}
\begin{equation}
\lim_{\varepsilon\to 0} \beta_\varepsilon = 1\ , \label{app1}
\end{equation}
and $B_\varepsilon\in L^p((0,\infty),z^m\mathrm{d}z)$ for all $(m,p)\in \mathcal{A}_\nu$ with
\begin{equation}
\lim_{\varepsilon\to 0} \int_0^1 z^m |B_\varepsilon(z)-B(z)|^p\ \mathrm{d}z = \lim_{\varepsilon\to 0} \int_0^1 z|\ln{z}|  |B_\varepsilon(z)-B(z)|\ \mathrm{d}z = 0 \ . \label{app2}
\end{equation}
\end{subequations}
An obvious consequence of \eqref{app2} is that
\begin{equation}
\lim_{\varepsilon\to 0} \mathfrak{b}_{m,p,\varepsilon} = \mathfrak{b}_{m,p} \ , \qquad (m,p)\in \mathcal{A}_\nu\ , \qquad\qquad \lim_{\varepsilon\to 0} \mathfrak{b}_{\ln,\varepsilon} = \mathfrak{b}_{\ln}\ , \label{app3} 
\end{equation}
where
\begin{equation*}
\mathfrak{b}_{m,p,\varepsilon} := \int_0^1 z^m B_\varepsilon(z)^p\ \mathrm{d}z\ , \qquad (m,p)\in \mathbb{R}\times [1,\infty)\ ,  \qquad\qquad \mathfrak{b}_{\ln,\varepsilon} := \int_0^1 z |\ln{z}| B_\varepsilon(z)\ \mathrm{d}z\ .
\end{equation*}
Recalling that $1+ \mathfrak{b}_{1+\lambda-\alpha,1}>2\mathfrak{b}_{1+\lambda-\alpha,1}$ due to $1+\lambda-\alpha>1$, it follows from \eqref{app1} and \eqref{app3} that there is $\varepsilon_0\in (0,1)$ such that, for $\varepsilon\in (0,\varepsilon_0)$, 
\begin{equation}
\mathfrak{b}_{m_0,1,\varepsilon} \le 1 + \mathfrak{b}_{m_0,1} \ , \qquad \mathfrak{b}_{1+\lambda-\alpha,1,\varepsilon} \le \frac{1+\mathfrak{b}_{1+\lambda-\alpha,1}}{2} < 1 \ , \qquad \mathfrak{b}_{m_1,q_1,\varepsilon} \le 1 + \mathfrak{b}_{m_1,q_1}\ . \label{app4}
\end{equation} 
An immediate consequence of \eqref{app4} is that, for $\varepsilon\in (0,\varepsilon_0)$,
\begin{equation}
\sup_{m\ge m_0}\{ \mathfrak{b}_{m,1,\varepsilon} \} \le 1 + \mathfrak{b}_{m_0,1}\ , \qquad \sup_{m\ge 1+\lambda-\alpha}\{ \mathfrak{b}_{m,1,\varepsilon} \} \le \frac{1+\mathfrak{b}_{1+\lambda-\alpha,1}}{2} \ . \label{app4a}
\end{equation}
Morever,
\begin{equation}
B_\varepsilon(z)=0\ , \qquad z\in [0,\varepsilon-\varepsilon^2]\ , \qquad B_\varepsilon\in W^{1,1}(0,1)\ , \label{app5}
\end{equation}
and
\begin{equation}
\int_0^1 B_\varepsilon(z)\ \mathrm{d}z \le \frac{1}{\varepsilon \beta_\varepsilon}\ , \qquad \sup_{z\in [0,1]}\{ B_\varepsilon(z)\} \le \int_0^1 \left| \frac{\mathrm{d}B_\varepsilon}{\mathrm{d}z}(z) \right|\ \mathrm{d}z \le \frac{1}{\varepsilon^3 \beta_\varepsilon} \int_{\mathbb{R}} \left| \frac{\mathrm{d}\zeta}{\mathrm{d}z}(z) \right|\ \mathrm{d}z\ . \label{app6}
\end{equation}

%%%%%%%%%%%%%%%%
\begin{remark}\label{remc}
In fact, if the function $B$ in \eqref{aa5} satisfies \eqref{aa6} for any $(m,p)\in\mathcal{A}_0$, as well as $B(0)=0$ and $B\in W^{1,1}(0,1)$, then we may take $B_\varepsilon=B$. This is true in particular for the parabolic daughter distribution function corresponding to $B(z) = 12 z(1-z)$, $z\in (0,1)$. 
\end{remark}
%%%%%%%%%%%%%%%%

Next, since $\varrho\in (0,\varrho_\star)$, we infer from \eqref{app3} that there is $\varepsilon_\varrho\in (0,\varepsilon_0)$ such that
\begin{equation}
\varrho < \frac{\varrho+\varrho_\star}{2} \le \varrho_{\star,\varepsilon} := \frac{a_0 \mathfrak{b}_{\ln,\varepsilon}}{2 K_0 \ln{2}}\ , \qquad \varepsilon\in (0,\varepsilon_\varrho)\ . \label{app7}
\end{equation}

Finally, since $m_1+\lambda-1\in (m_0,\lambda)$ by \eqref{aa1} and $(m_1,1)\in \mathcal{A}_\nu$ by \eqref{aa12a}, we may fix 
\begin{equation}
q_1\in (1,2) \;\text{ such that }\; (m_1,q_1)\in \mathcal{A}_\nu \;\text{ and }\; \frac{m_1+1+q_1(\lambda-2)}{q_1} \in (m_0,\lambda)\ . \label{p1} 
\end{equation}

The main result of this section is then the following:

%%%%%%%%%%%%%%%%
\begin{proposition}\label{PropApp}
Let $\varepsilon\in (0,\varepsilon_\varrho)$. There is 
\begin{equation*}
\varphi_\varepsilon\in X_1^+ \cap L^{q_1}((0,\infty),x^{m_1}\mathrm{d}x) \cap W^{1,1}(0,\infty) \cap \bigcap_{m\ge \lambda-2} X_m\ , 
\end{equation*}
such that $M_1(\varphi_\varepsilon)=\varrho$ and
\begin{align}
\int_0^\infty \left[ \vartheta(x) - x \partial_x\vartheta(x) \right] \varphi_\varepsilon(x)\ \mathrm{d}x & = \frac{1}{2} \int_0^\infty \int_0^\infty K(x,y) \chi_\vartheta(x,y) \varphi_\varepsilon(x) \varphi_\varepsilon(y)\ \mathrm{d}y\mathrm{d}x \nonumber \\
& \qquad - \int_0^\infty a(x) N_{\vartheta,\varepsilon}(x) \varphi_\varepsilon(x)\ \mathrm{d}x\ , \label{app8}
\end{align}
for all $\vartheta\in \Theta_1$, where $\Theta_1$ is defined in \eqref{a6c} and
\begin{equation*}
N_{\vartheta,\varepsilon}(y) := \vartheta(y) - \int_0^y \vartheta(x) b_\varepsilon(x,y)\ \mathrm{d}x\ , \qquad y>0\ .
\end{equation*}
Moreover, 
\begin{subequations}\label{x0}
\begin{itemize}
\item[(a)] There is $\ell>0$ depending only on $\lambda$, $\alpha$, $K_0$, $a_0$, $B$, $\nu$, $m_0$, $m_1$, $q_1$, and $\varrho$ such that
\begin{align}
\int_0^\infty x \ln{(x)}\; \varphi_\varepsilon(x)\ \mathrm{d}x + \frac{3}{e(1-m_1)} M_{m_1}(\varphi_\varepsilon) & \le \ell\ , \label{x1}\\
M_{m_0}(\varphi_\varepsilon) & \le \ell\ , \label{x2} \\
\int_0^\infty x^{m_1} \varphi_\varepsilon(x)^{q_1}\ \mathrm{d}x & \le \ell \ . \label{x3}
\end{align}
\item[(b)] For all $m\ge 1+\lambda$, there is $L(m)>0$ depending only on $\lambda$, $\alpha$, $K_0$, $a_0$, $B$, $\nu$, $m_0$, $m_1$, $q_1$, $\varrho$, and $m$ such that
\begin{equation}
M_m(\varphi_\varepsilon) \le L(m)\ . \label{x4}
\end{equation}
\end{itemize}
\end{subequations}
\end{proposition}
%%%%%%%%%%%%%%%%

The main steps in the proof of Proposition~\ref{PropApp} are the derivation of \eqref{x1} and \eqref{x3}. The former is inspired from \cite[Lemma~4.2]{EMRR05} and combines a differential inequality for a superlinear moment, involving here the weight $x\mapsto x\ln{x}$, and a differential inequality for a sublinear moment. The validity of \eqref{x1} requires the smallness condition $\varrho\in (0,\varrho_\star)$, the value of $\varrho_\star$ being prescribed by an algebraic inequality established in \cite[Lemma~2.3]{LaurXX}, see \eqref{AIL} below. As for \eqref{x3}, it relies on the monotonicity of $x\mapsto x^{m_1} K(x,y)$ to handle the contribution of the coagulation term, similar arguments being used in \cite{BLLxx, Buro83, LaMi02c, MiRR03} to derive $L^p$-estimates for solutions to coagulation-fragmentation equations.

%%%%%%%%%%%%%%%%
%%%%%%%%%%%%%%%%
\subsection{Scaling variables and well-posedness} \label{sec2.2}
%%%%%%%%%%%%%%%%
%%%%%%%%%%%%%%%%

Let $\varepsilon\in (0,\varepsilon_\varrho)$. We begin with the existence and uniqueness of a mass-conserving weak solution to
\begin{subequations}\label{d3}
\begin{align}
\partial_s g_\varepsilon(s,x) & = - x \partial_x g_\varepsilon(s,x) - 2 g_\varepsilon(s,x) + \mathcal{C}g_\varepsilon(s,x) + \mathcal{F}_\varepsilon g_\varepsilon(s,x)\ , \qquad (s,x) \in (0,\infty)^2\ , \label{d3a} \\
g_\varepsilon(0,x) & = f^{in}(x)\ , \qquad x\in (0,\infty)\ , \label{d3d}
\end{align}
\end{subequations}
where $\mathcal{F}_\varepsilon$ denotes the fragmentation operator with $b$ replaced with $b_\varepsilon$.

%%%%%%%%%%%%%%%%
\begin{proposition}\label{PropP1}
Consider an initial condition $f^{in}\in X_1^+\cap X_{m_0}\cap X_{2\lambda-\alpha}$ such that 
\begin{equation}
M_1(f^{in}) = \varrho\ . \label{a100}
\end{equation}
There is a unique mass-conserving weak solution $g_\varepsilon$ to \eqref{a1} on $[0,\infty)$ satisfying
\begin{equation*}
g_\varepsilon\in C([0,T),X_{m_1,w})\cap L^\infty((0,T),X_{m_0}) \cap L^\infty((0,T),X_{2\lambda-\alpha}) \;\text{ for any }\; T>0\ , 
\end{equation*} 
\begin{equation}
M_1(g_\varepsilon(s)) = \varrho\ , \qquad s\ge 0\ , \label{d9}
\end{equation}
and
\begin{align}
\int_0^\infty (g_\varepsilon(s,x)-f^{in}(x)) \vartheta(x)\ \mathrm{d}x & = \int_0^s \int_0^\infty \left[ x \partial_x\vartheta(x) - \vartheta(x) \right] g_\varepsilon(s_*,x)\ \mathrm{d}x\mathrm{d}s_* \nonumber\\
& \qquad + \frac{1}{2} \int_0^s \int_0^\infty \int_0^\infty K(x,y) \chi_\vartheta(x,y) g_\varepsilon(s_*,x) g_\varepsilon(s_*,y)\ \mathrm{d}y\mathrm{d}x\mathrm{d}s_* \nonumber \\
& \qquad - \int_0^s \int_0^\infty a(y) N_{\vartheta,\varepsilon}(y) g_\varepsilon(s_*,y)\ \mathrm{d}y\mathrm{d}s_*\ , \label{d3b}
\end{align}
for all $s\in (0,\infty)$ and $\vartheta\in\Theta_{m_1}$, where $\Theta_0:= L^\infty(0,\infty)$ and 
\begin{equation*}
\Theta_m := \left\{ \vartheta\in C^m([0,\infty))\cap L^\infty(0,\infty)\ :\ \vartheta(0) = 0 \right\}\ .
\end{equation*} 
We recall that $N_{\vartheta,\varepsilon}$ in \eqref{d3b} is defined in Proposition~\ref{PropApp}, 
\end{proposition}
%%%%%%%%%%%%%%%%

\begin{proof}
Owing to \eqref{aa1}, \eqref{aa2}, \eqref{aa3}, \eqref{be3}, \eqref{app0}, and the integrability properties of $B_\varepsilon$, we are in a position to apply \cite[Theorem~1.2]{LaurXX}, which guarantees the existence and uniqueness of a mass-conserving weak solution $f_\varepsilon$ to the coagulation-fragmentation equation 
\begin{subequations}\label{dd3}
\begin{align}
\partial_t f_\varepsilon(t,x) & = \mathcal{C}f_\varepsilon(t,x) + \mathcal{F}_\varepsilon f_\varepsilon(t,x)\ , \qquad (t,x) \in (0,\infty)^2\ , \label{dd3a} \\
f_\varepsilon(0,x) & = f^{in}(x)\ , \qquad x\in (0,\infty)\ , \label{dd3d}
\end{align}
\end{subequations}
which satisfies 
\begin{equation*}
f_\varepsilon\in C([0,T),X_{m_1,w})\cap L^\infty((0,T),X_{m_0}) \cap L^\infty((0,T),X_{2\lambda-\alpha}) 
\end{equation*} 
for any $T>0$ and $M_1(f_\varepsilon(t)) = \varrho$ for $t\ge 0$. Setting 
\begin{equation}
\Psi_\varepsilon(s;f^{in})(x) = g_\varepsilon(s,x) := e^{-2s} f_\varepsilon\left( \frac{e^{(\lambda-1)s}-1}{\lambda-1} , x e^{-s} \right)\ , \qquad (s,x)\in  [0,\infty) \times (0,\infty)\ , \label{d2}
\end{equation}
completes the proof of Proposition~\ref{PropP1}.
\end{proof}

The next results are devoted to the derivation of a series of estimates satisfied by the weak solutions to \eqref{d3} provided by Proposition~\ref{PropP1}, except for Lemma~\ref{Lemd8} where the continuous dependence of $\Psi_\varepsilon(\cdot;f^{in})$ in $X_1$ with respect to the initial condition is established.

Throughout the remainder of this section, $\kappa$ and $(\kappa_i)_{i\ge 1}$ are positive constants depending only on $\lambda$, $\alpha$, $K_0$, $a_0$, $B$, $\nu$, $m_0$, $m_1$, $q_1$, and $\varrho$. Dependence upon additional parameters is indicated explicitly.

\newcounter{Num2Const}

%%%%%%%%%%%%%%%%
%%%%%%%%%%%%%%%%
\subsection{Moment Estimates} \label{sec3.1}
%%%%%%%%%%%%%%%%
%%%%%%%%%%%%%%%%

We begin with the derivation of estimates for moments of order $m\in [m_1,1]$, the parameter $m_1$ being defined in \eqref{aa11}.

%%%%%%%%%%%%%%%%
\begin{lemma}\label{Lemd1}
\refstepcounter{Num2Const}\label{kap1}
Consider $\varepsilon\in (0,\varepsilon_\varrho)$ and $f^{in}\in X_0^+\cap X_{1+\lambda}$ such that $M_1(f^{in})=\varrho$ and let $g_\varepsilon=\Psi_\varepsilon(\cdot;f^{in})$ be given by \eqref{d2}. For $m\in [m_1,1)$, there is $\kappa_{\ref{kap1}}(m)>0$ depending on $m$ such that, for $t\ge 0$, 
\begin{align*}
& \int_0^\infty x\ln{(x)}\, g_\varepsilon(s,x)\ \mathrm{d}x + \frac{3}{e (1-m)} M_m(g_\varepsilon(s)) \\
& \qquad\qquad \le \max\left\{ \int_0^\infty x\ln{(x)}\, f^{in}(x)\ \mathrm{d}x + \frac{3}{e (1-m)} M_m(f^{in}) , \kappa_{\ref{kap1}}(m) \right\}\ , \qquad s\ge 0\ .
\end{align*}
\end{lemma}
%%%%%%%%%%%%%%%%

\begin{proof}
Let $s\ge 0$ and consider $m\in [m_1,1)$. Then
\begin{equation*}
\chi_m(x,y) := (x+y)^m - x^m - y^m \le 0 \ , \qquad (x,y)\in (0,\infty)^2\ , 
\end{equation*}
and 
\begin{equation*}
N_{m,\varepsilon}(y) := y^m - \int_0^y x^m b_\varepsilon(x,y)\ \mathrm{d}x = (1-\mathfrak{b}_{m,1,\varepsilon}) y^m \ge -\mathfrak{b}_{m,1,\varepsilon} y^m\ , \qquad y\in (0,\infty)\ .
\end{equation*}
Consequently, we infer from \eqref{aa2}, \eqref{app4a}, \eqref{d3b} (with $\vartheta(x)=x^m$, $x>0$), and the non-negativity of $g_\varepsilon$ and $K$ that
\begin{align*}
\frac{\mathrm{d}}{\mathrm{d}s} M_m(g_\varepsilon(s)) & \le -(1-m) M_m(g_\varepsilon(s)) + a_0 \mathfrak{b}_{m,1,\varepsilon} M_{m+\lambda-1}(g_\varepsilon(s)) \\
& \le -(1-m) M_m(g_\varepsilon(s)) + a_0 (1+\mathfrak{b}_{m_0,1}) M_{m+\lambda-1}(g_\varepsilon(s))\ .
\end{align*}
Observing that $m+\lambda-1\in [1,\lambda)$, it follows from \eqref{d9} and H\"older's inequality that
\begin{align*}
M_{m+\lambda-1}(g_\varepsilon(s)) & \le M_{\lambda}(g_\varepsilon(s))^{(m+\lambda-2)/(\lambda-1)} M_{1}(g_\varepsilon(s))^{(1-m)/(\lambda-1)} \\ 
& \le \varrho^{(1-m)/(\lambda-1)} M_{\lambda}(g_\varepsilon(s))^{(m+\lambda-2)/(\lambda-1)}\ .
\end{align*}
We combine the previous two inequalities and use Young's inequality (since $m+\lambda-2<\lambda-1$) to obtain
\begin{equation}
\frac{\mathrm{d}}{\mathrm{d}s} M_m(g_\varepsilon(s)) \le -(1-m) M_m(g_\varepsilon(s)) + \frac{e(1-m)}{3} \delta_\varrho M_\lambda(g_\varepsilon(s)) + \frac{e(1-m)}{3} \kappa(m)\ , \label{xx1}
\end{equation}
with
\begin{equation}
\delta_\varrho : = \frac{K_0 \ln{2}}{2} (\varrho_\star-\varrho)  > 0\ , \label{Drho}
\end{equation}
We next set $\bar{\vartheta}(x) = x \ln{x}$ for $x\ge 0$ and recall the inequality 
\begin{equation}
\chi_{\bar{\vartheta}}(x,y) = (x+y) \ln{(x+y)} - x \ln{x} -y\ln{y} \le 2 \ln{2} \sqrt{xy}\ , \qquad (x,y)\in (0,\infty)^2\ , \label{AIL}
\end{equation}
established in \cite[Lemma~2.3]{LaurXX}, along with the following consequence of \eqref{aa1}, \eqref{aa3}, and Young's inequality
\begin{align*}
\sqrt{xy} K(x,y) & \le K_0 xy \left( x^{(2\alpha-1)/2} y^{(2\lambda-2\alpha-1)/2} + x^{(2\lambda-2\alpha-1)/2} y^{(2\alpha-1)/2}\right) \\
& \le K_0xy \left( \frac{2\alpha-1}{2(\lambda-1)} x^{\lambda-1} + \frac{2\lambda-2\alpha-1}{2(\lambda-1)} y^{\lambda-1} + \frac{2\lambda-2\alpha-1}{2(\lambda-1)} x^{\lambda-1} + \frac{2\alpha-1}{2(\lambda-1)} y^{\lambda-1} \right) \\
& \le K_0 \left( x^\lambda y + x y^\lambda \right)\ , \qquad (x,y)\in (0,\infty)^2\ .
\end{align*}
Also, by \eqref{be3} and \eqref{app0}, 
\begin{equation*}
N_{\bar{\vartheta},\varepsilon}(y) = y \ln{y} - \int_0^1 yz \ln{(yz)} B_\varepsilon(z)\ \mathrm{d}z = \mathfrak{b}_{\ln,\varepsilon} y\ ,  \qquad y\in (0,\infty)\ .
\end{equation*}
We then infer from \eqref{aa2}, \eqref{aa3},  \eqref{app7}, \eqref{d9}, and \eqref{d3b} (with $\vartheta=\bar{\vartheta}$) that 
\begin{align*}
\frac{\mathrm{d}}{\mathrm{d}s} \int_0^\infty \bar{\vartheta}(x) g_\varepsilon(s,x)\ \mathrm{d}x & \le M_1(g_\varepsilon(s)) + 2K_0 \ln{(2)} M_1(g_\varepsilon(s)) M_\lambda(g_\varepsilon(s)) - a_0 \mathfrak{b}_{\ln,\varepsilon} M_\lambda(g_\varepsilon(s)) \\
& \le \varrho + 2K_0 \ln{2} \left( \varrho - \varrho_{\star,\varepsilon} \right) M_\lambda(g_\varepsilon(s)) \\
& \le  \varrho - 2\delta_\varrho M_\lambda(g_\varepsilon(s))\ ,
\end{align*}
the parameter $\delta_\varrho$ being defined in \eqref{Drho}. Combining \eqref{xx1} and the previous inequality, we find \refstepcounter{Num2Const}\label{kap1.5}
\begin{equation}
\frac{\mathrm{d}}{\mathrm{d}s} U_{m,\varepsilon}(s)  + \frac{3}{e} M_m(g_\varepsilon(s)) + \delta_\varrho M_\lambda(g_\varepsilon(s)) \le \kappa_{\ref{kap1.5}}(m)\ , \label{d10}
\end{equation}
where
\begin{equation*}
U_{m,\varepsilon}(s) := \int_0^\infty \bar{\vartheta}(x) g_\varepsilon(s,x)\ \mathrm{d}x + \frac{3}{e(1-m)} M_m(g_\varepsilon(s))\ .
\end{equation*}
Since 
\begin{equation*}
x^{\lambda-1} \ge \ln{x} + \frac{1+\ln{(\lambda-1)}}{\lambda-1}\ , \qquad x\in (0,\infty)\ ,
\end{equation*}
there holds
\begin{equation*}
M_\lambda(g_\varepsilon(s)) \ge \int_0^\infty \bar{\vartheta}(x) g_\varepsilon(s,x)\ \mathrm{d}x + \frac{1+\ln{(\lambda-1)}}{\lambda-1} M_1(g_\varepsilon(s))\ .
\end{equation*}
\refstepcounter{Num2Const}\label{kap2} \refstepcounter{Num2Const}\label{kap2.5}
Consequently, setting $\kappa_{\ref{kap2}}(m) := \min\{ 1-m , \delta_\varrho \}$ and using once more \eqref{d9}, we obtain
\begin{align*}
\frac{\mathrm{d}}{\mathrm{d}s} U_{m,\varepsilon}(s) & + \kappa_{\ref{kap2}}(m) U_{m,\varepsilon}(s) \\
& \le \frac{\mathrm{d}}{\mathrm{d}s} U_{m,\varepsilon}(s) + \kappa_{\ref{kap2}}(m) \left[ M_\lambda(g_\varepsilon(s)) - \varrho\frac{1+\ln{(\lambda-1)}}{\lambda-1} + \frac{3}{e(1-m)} M_m(g_\varepsilon(s)) \right] \\
& \le \left( \kappa_{\ref{kap2}}(m) - \delta_\varrho \right) M_\lambda(g_\varepsilon(s)) + \frac{3 \left[ \kappa_{\ref{kap2}}(m) - (1-m) \right]}{e(1-m)} M_m(g_\varepsilon(s)) + \kappa_{\ref{kap2.5}}(m) \\
& \le \kappa_{\ref{kap2.5}}(m)\ . 
\end{align*}
Integrating with respect to $s$ gives
\begin{equation*}
U_{m,\varepsilon}(s) \le e^{-\kappa_{\ref{kap2}}(m)s} U_{m,\varepsilon}(0) + \frac{\kappa_{\ref{kap2.5}}(m)}{\kappa_{\ref{kap2}}(m)} \left( 1- e^{-\kappa_{\ref{kap2}}(m)s} \right) \le  \max\left\{ U_{m,\varepsilon}(0) , \frac{\kappa_{\ref{kap2.5}}(m)}{\kappa_{\ref{kap2}}(m)}\right\} 
\end{equation*}
for $s\ge 0$ and Lemma~\ref{Lemd1} follows with $\kappa_{\ref{kap1}}(m) := \kappa_{\ref{kap2.5}}(m)/\kappa_{\ref{kap2}}(m)$. 
\end{proof}

From now on, we assume that $f^{in}$ satisfies
\begin{equation}
M_1(f^{in}) = \varrho \;\text{ and }\; \int_0^\infty x\ln{(x)}\, f^{in}(x)\ \mathrm{d}x + \frac{3}{e (1-m_1)} M_{m_1}(f^{in}) \le \kappa_{\ref{kap1}}(m_1)\ . \label{d11}
\end{equation}
A straightforward consequence of Lemma~\ref{Lemd1} is the following estimate.

%%%%%%%%%%%%%%%%
\begin{corollary}\label{Cord2}
\refstepcounter{Num2Const}\label{kap3}
Consider $\varepsilon\in (0,\varepsilon_\varrho)$ and $f^{in}\in X_0^+\cap X_{1+\lambda}$ satisfying \eqref{d11} and let $g_\varepsilon=\Psi_\varepsilon(\cdot;f^{in})$ be given by \eqref{d2}. There is $\kappa_{\ref{kap3}}>0$ such that 
\begin{equation*}
\int_0^\infty x|\ln{x}| g_\varepsilon(s,x)\ \mathrm{d}x + M_{m_1}(g_\varepsilon(s)) \le \kappa_{\ref{kap3}}\ , \qquad s\ge 0\ .
\end{equation*}
\end{corollary}
%%%%%%%%%%%%%%%%

\begin{proof}
Let $s\ge 0$. Since
\begin{equation*}
x |\ln{x}| - \frac{2 x^{m_1}}{e(1-m_1)} \le x\ln{x} \le x |\ln{x}|\ , \qquad x>0\ , 
\end{equation*}
it follows from \eqref{d11} and Lemma~\ref{Lemd1} (with $m=m_1$) that
\begin{align*}
& \int_0^\infty x|\ln{x}| g_\varepsilon(s,x)\ \mathrm{d}x + \frac{1}{e(1-m_1)} M_{m_1}(g_\varepsilon(s)) \nonumber \\
& \qquad \le \int_0^\infty x\ln{x}\; g_\varepsilon(s,x)\ \mathrm{d}x + \frac{3}{e(1-m_1)} M_{m_1}(g_\varepsilon(s)) \nonumber \\
& \qquad \le \kappa_{\ref{kap1}}(m_1)\ ,
\end{align*}
from which Corollary~\ref{Cord2} follows.
\end{proof}

Thanks to Corollary~\ref{Cord2}, we may derive additional information on the behaviour of $g_\varepsilon$ for large sizes.

%%%%%%%%%%%%%%%%
\begin{lemma}\label{Lemd3}
\refstepcounter{Num2Const}\label{kap4}
Consider $\varepsilon\in (0,\varepsilon_\varrho)$ and $f^{in}\in X_0^+\cap X_{1+\lambda}$ satisfying \eqref{d11} and let $g_\varepsilon=\Psi_\varepsilon(\cdot;f^{in})$ be given by \eqref{d2}. Assume also that $f^{in}\in X_m$ for some $m>1+\lambda-\alpha$. Then there is $\kappa_{\ref{kap4}}(m)>0$ depending on $m$ such that 
\begin{equation*}
M_m(g_\varepsilon(s)) \le \max\left\{ M_m(f^{in}) , \kappa_{\ref{kap4}}(m) \right\}\ , \qquad s\ge 0\ .
\end{equation*}
\end{lemma}
%%%%%%%%%%%%%%%%

\begin{proof}
Let $s\ge 0$. We infer from \eqref{app} and \eqref{d3b} that 
\begin{equation}
\frac{\mathrm{d}}{\mathrm{d}s} M_m(g_\varepsilon(s)) = (m-1) M_m(g_\varepsilon(s)) + P_{m,\varepsilon}(s) - a_0 (1-\mathfrak{b}_{m,1,\varepsilon}) M_{m+\lambda-1}(g_\varepsilon(s))\ , \label{xx2}
\end{equation}
with
\begin{equation*}
P_{m,\varepsilon}(s) := \frac{1}{2} \int_0^\infty \int_0^\infty K(y,y_*) \chi_m(y,y_*) g_\varepsilon(s,y) g_\varepsilon(s,y_*)\ \mathrm{d}y_* \mathrm{d}y\ .
\end{equation*}
On the one hand, since $\lambda>1$, it follows from \eqref{d9}, \eqref{d11}, and H\"older's inequality that
\begin{equation*}
M_m(g_\varepsilon(s)) \le M_{m+\lambda-1}(g_\varepsilon(s))^{(m-1)/(m+\lambda-2)} \varrho^{(\lambda-1)/(m+\lambda-2)}\ .
\end{equation*}
Equivalently,
\begin{equation*}
\varrho^{(1-\lambda)/(m-1)} M_m(g_\varepsilon(s))^{(m+\lambda-2)/(m-1)} \le M_{m+\lambda-1}(g_\varepsilon(s))\ .
\end{equation*}
In addition, by \eqref{app4a}, 
\begin{equation*}
1-\mathfrak{b}_{m,1,\varepsilon} \ge \frac{1-\mathfrak{b}_{1+\lambda-\alpha,1}}{2} > 0 \ .
\end{equation*}
Consequently,
\begin{equation}
- a_0 (1-\mathfrak{b}_{m,1,\varepsilon}) M_{m+\lambda-1}(g_\varepsilon(s)) \le - 4\delta_{\varrho,m} M_m(g_\varepsilon(s))^{(m+\lambda-2)/(m-1)}\ , \label{xx3}
\end{equation}
with
\begin{equation}
\delta_{\varrho,m} := \frac{a_0 (1-\mathfrak{b}_{1+\lambda-\alpha,1}) \varrho^{(1-\lambda)/(m-1)}}{8} > 0\ . \label{xx4}
\end{equation}
On the other hand, to estimate the contribution of the coagulation term, we argue as in \cite[Lemma~2.6]{LaurXX}. Since $m>1$, there is $c_m>0$ depending only on $m$ such that 
\begin{equation*}
\chi_m(x,y) = (x+y)^m -x^m-y^m \le c_m \left( xy^{m-1} + x^{m-1} y \right)\ , \qquad (x,y)\in (0,\infty)^2\ ,
\end{equation*}
and it follows from \eqref{aa3} and the previous inequality that
\begin{align*}
P_{m,\varepsilon}(s) & \le \frac{c_m}{2} \int_0^\infty \int_0^\infty K(x,y) \left( xy^{m-1} + x^{m-1} y \right) g_\varepsilon(s,x) g_\varepsilon(s,y)\ \mathrm{d}y \mathrm{d}x \\
& = K_0 c_m \int_0^\infty \int_0^\infty xy^{m-1} \left( x^\alpha y^{\lambda-\alpha} + x^{\lambda-\alpha} y^\alpha \right) g_\varepsilon(s,x) g_\varepsilon(s,y)\ \mathrm{d}y \mathrm{d}x \\
& = K_0 c_m \left[ M_{1+\alpha}(g_\varepsilon(s)) M_{m+\lambda-\alpha-1}(g_\varepsilon(s)) + M_{1+\lambda-\alpha}(g_\varepsilon(s)) M_{m+\alpha-1}(g_\varepsilon(s)) \right]\ .
\end{align*}
Owing to \eqref{aa1} and $m>1+\lambda-\alpha\ge 1+\alpha$, both $m+\lambda-\alpha-1$ and $m+\alpha-1$ belong to $[1,m]$ and we deduce from \eqref{d9}, \eqref{d11}, and H\"older's inequality that
\begin{align*}
M_{m+\lambda-\alpha-1}(g_\varepsilon(s)) & \le \varrho^{(1+\alpha-\lambda)/(m-1)} M_m(g_\varepsilon(s))^{(m+\lambda-\alpha-2)/(m-1)}\ , \\
M_{m+\alpha-1}(g_\varepsilon(s)) & \le \varrho^{(1-\alpha)/(m-1)} M_m(g_\varepsilon(s))^{(m+\alpha-2)/(m-1)}\ .
\end{align*}
Also, introducing
\begin{equation*}
Q_\varepsilon(s,R) := \int_R^\infty y g_\varepsilon(s,y)\ \mathrm{d}y\ , \qquad R>1\ ,
\end{equation*}
and noticing that $1<1+\alpha \le 1+\lambda-\alpha<m$, we infer from \eqref{d9}, \eqref{d11}, and H\"older's inequality that, for $R>1$,
\begin{align*}
M_{1+\alpha}(g_\varepsilon(s)) & \le R^\alpha \int_0^R x g_\varepsilon(s,x)\ \mathrm{d}x \\
& \qquad + \left( \int_R^\infty x^m g_\varepsilon(s,x)\ \mathrm{d}x \right)^{\alpha/(m-1)} \left( \int_R^\infty x g_\varepsilon(s,x)\ \mathrm{d}x \right)^{(m-1-\alpha)/(m-1)} \\
& \le R \varrho + Q_\varepsilon(s,R)^{(m-1-\alpha)/(m-1)} M_m(g_\varepsilon(s))^{\alpha/(m-1)} \\
& \le R \varrho + \varrho^{(\lambda-2\alpha)/(m-1)} Q_\varepsilon(s,R)^{(m+\alpha-\lambda-1)/(m-1)} M_m(g_\varepsilon(s))^{\alpha/(m-1)}
\end{align*}
and
\begin{align*}
M_{1+\lambda-\alpha}(g_\varepsilon(s)) & \le R^{\lambda-\alpha} \int_0^R x g_\varepsilon(s,x)\ \mathrm{d}x \\
& \qquad + \left( \int_R^\infty x^m g_\varepsilon(s,x)\ \mathrm{d}x \right)^{(\lambda-\alpha)/(m-1)} \left( \int_R^\infty x g_\varepsilon(s,x)\ \mathrm{d}x \right)^{(m+\lambda-1-\alpha)/(m-1)} \\
& \le R \varrho + Q_\varepsilon(s,R)^{(m+\alpha-\lambda-1)/(m-1)} M_m(g_\varepsilon(s))^{(\lambda-\alpha)/(m-1)}\ .
\end{align*}
Collecting the above estimates, we find \refstepcounter{Num2Const}\label{kap5} 
\begin{align*}
P_{m,\varepsilon}(s) & \le \kappa_{\ref{kap5}}(m) R \left[ M_m(g_\varepsilon(s))^{(m+\alpha-2)/(m-1)} + M_m(g_\varepsilon(s))^{(m+\lambda-\alpha-2)/(m-1)} \right] \\
& \qquad + \kappa_{\ref{kap5}}(m) Q_\varepsilon(s,R)^{(m-1-\lambda+\alpha)/(m-1)} M_m(g_\varepsilon(s))^{(m+\lambda-2)/(m-1)}
\end{align*}
for $R>1$. Owing to Corollary~\ref{Cord2}, 
\begin{equation*}
Q_\varepsilon(s,R) \le \frac{1}{\ln{R}} \int_R^\infty y |\ln{y}| g_\varepsilon(s,y)\ \mathrm{d}y \le \frac{\kappa_{\ref{kap3}}}{\ln{R}}\ .
\end{equation*}
Introducing $R_m>1$ defined by
\begin{equation*}
\kappa_{\ref{kap5}}(m) \left( \frac{\kappa_{\ref{kap3}}}{\ln{R_m}} \right)^{(m-1-\lambda+\alpha)/(m-1)} = \delta_{\varrho,m}
\end{equation*}
and taking $R=R_m$ in the previous estimate on $P_{m,\varepsilon}(s)$ give
\begin{align*}
P_{m,\varepsilon}(s) & \le \kappa_{\ref{kap5}}(m) R_m \left[ M_m(g_\varepsilon(s))^{(m+\alpha-2)/(m-1)} + M_m(g_\varepsilon(s))^{(m+\lambda-\alpha-2)/(m-1)} \right] \\
& \qquad + \delta_{\varrho,m} M_m(g_\varepsilon(s))^{(m+\lambda-2)/(m-1)}\ .
\end{align*}
Since $m+\alpha-2<m+\lambda-2$ and $m+\lambda-\alpha-2<m+\lambda-2$, we apply Young's inequality to obtain
\begin{equation}
P_{m,\varepsilon}(s) \le \kappa(m) + 2\delta_{\varrho,m} M_m(g_\varepsilon(s))^{(m+\lambda-2)/(m-1)}\ . \label{xx5}
\end{equation}
We now combine \eqref{xx2}, \eqref{xx3}, and \eqref{xx5} and obtain
\begin{equation*}
\frac{\mathrm{d}}{\mathrm{d}s} M_m(g_\varepsilon(s)) \le \kappa(m) + (m-1) M_m(g_\varepsilon(s)) - 2\delta_{\varrho,m} M_m(g_\varepsilon(s))^{(m+\lambda-2)/(m-1)} \ .
\end{equation*}
Hence, using once more Young's inequality, \refstepcounter{Num2Const}\label{kap6}
\begin{align*}
\frac{\mathrm{d}}{\mathrm{d}s} M_m(g_\varepsilon(s)) & \le \kappa_{\ref{kap6}}(m) - \delta_{\varrho,m} M_m(g_\varepsilon(s))^{(m+\lambda-2)/(m-1)} \\
& = \delta_{\varrho,m} \left[  \kappa_{\ref{kap4}}(m)^{(m+\lambda-2)/(m-1)} - M_m(g_\varepsilon(s))^{(m+\lambda-2)/(m-1)} \right] \ ,
\end{align*}
with $\kappa_{\ref{kap4}}(m) := (\kappa_{\ref{kap6}}(m)/\delta_{\varrho,m})^{(m-1)/(m+\lambda-2)}$. Lemma~\ref{Lemd3} is then a consequence of the comparison principle.
\end{proof}

We finally return to the behaviour for small sizes.

%%%%%%%%%%%%%%%%
\begin{lemma}\label{Lemd4}
\refstepcounter{Num2Const}\label{kap7}
Consider $\varepsilon\in (0,\varepsilon_\varrho)$ and $f^{in}\in X_0^+\cap X_{1+\lambda}$ satisfying \eqref{d11} and let $g_\varepsilon=\Psi_\varepsilon(\cdot;f^{in})$ be given by \eqref{d2}. For $m\in [m_0,m_1)$, there is $\kappa_{\ref{kap7}}(m)>0$ depending on $m$ such that, if $f^{in}\in X_m$, then
\begin{equation*}
M_m(g_\varepsilon(s)) \le \max\left\{ M_m(f^{in}) , \kappa_{\ref{kap7}}(m) \mathcal{M}_{1+\lambda,\varepsilon} \right\}\ , \qquad s\ge 0\ ,
\end{equation*}
where
\begin{equation*}
\mathcal{M}_{1+\lambda,\varepsilon} := \sup_{s\ge 0}\{ M_{1+\lambda}(g_\varepsilon(s)) \} < \infty\ .
\end{equation*}
\end{lemma}
%%%%%%%%%%%%%%%%

\begin{proof}
We first note that $\mathcal{M}_{1+\lambda,\varepsilon}$ is indeed finite according to Lemma~\ref{Lemd3}. Next, let $s\ge 0$. As at the beginning of the proof of Lemma~\ref{Lemd1}, we infer from \eqref{app}, \eqref{app4a}, and \eqref{d3} that 
\begin{align*}
\frac{\mathrm{d}}{\mathrm{d}s} M_m(g_\varepsilon(s)) & \le -(1-m) M_m(g_\varepsilon(s)) + a_0 \mathfrak{b}_{m,1,\varepsilon} M_{m+\lambda-1}(g_\varepsilon(s))\\ 
&\le -(1-m) M_m(g_\varepsilon(s)) + a_0 (1+\mathfrak{b}_{m_0,1}) M_{m+\lambda-1}(g_\varepsilon(s))\ .
\end{align*}
Since $m+\lambda-1\in (m,1+\lambda)$, we deduce from H\"older's inequality that
\begin{equation*}
M_{m+\lambda-1}(g_\varepsilon(s)) \le M_{1+\lambda}(g_\varepsilon(s))^{(\lambda-1)/(1+\lambda-m)} M_m(g_\varepsilon(s))^{(2-m)/(1+\lambda-m)}\ .
\end{equation*}
Consequently, 
\begin{align*}
\frac{\mathrm{d}}{\mathrm{d}s} M_m(g_\varepsilon(s)) & \le -(1-m) M_m(g_\varepsilon(s)) + a_0 (1+\mathfrak{b}_{m_0,1}) \mathcal{M}_{1+\lambda,\varepsilon}^{(\lambda-1)/(1+\lambda-m)} M_m(g_\varepsilon(s))^{(2-m)/(1+\lambda-m)} \\
& = (1-m) M_m(g_\varepsilon(s))^{(2-m)/(1+\lambda-m)} \left\{ \left[ \kappa_{\ref{kap7}}(m) \mathcal{M}_{1+\lambda,\varepsilon} \right]^{(\lambda-1)/(1+\lambda-m)} \right. \\
& \hspace{8cm} \left. - M_m(g_\varepsilon(s))^{(\lambda-1)/(1+\lambda-m)} \right\}\ ,
\end{align*}
with $\kappa_{\ref{kap7}}(m) := (a_0 (1+\mathfrak{b}_{m_0,1})/(1-m))^{(1+\lambda-m)/(\lambda-1)}$. Lemma~\ref{Lemd4} follows from the above differential inequality and the comparison principle.
\end{proof}

Up to now, we have derived estimates which do not depend on $\varepsilon\in (0,\varepsilon_\varrho)$ and which will thus be of utmost importance in the next section to take the limit $\varepsilon\to 0$. However, these estimates do not provide enough control on the behaviour for small sizes for the proof of Proposition~\ref{PropApp}, for which the next result is required.

%%%%%%%%%%%%%%%%
\begin{lemma}\label{Lemd100}
\refstepcounter{Num2Const}\label{kap100}
Consider $\varepsilon\in (0,\varepsilon_\varrho)$ and $f^{in}\in X_0^+\cap X_{1+\lambda}$ satisfying \eqref{d11} and let $g_\varepsilon=\Psi_\varepsilon(\cdot;f^{in})$ be given by \eqref{d2}. For $m\in (-1,0]$, there is $\kappa_{\ref{kap100}}(m,\varepsilon)>0$ depending on $m$ and $\varepsilon$ such that
\begin{equation*}
M_m(g_\varepsilon(s)) \le \max\left\{ M_m(f^{in}) , \kappa_{\ref{kap100}}(m,\varepsilon) \mathcal{M}_{1+\lambda,\varepsilon} \right\}\ . \qquad s\ge 0\ .
\end{equation*}
\end{lemma}
%%%%%%%%%%%%%%%%

\begin{proof}
The proof is exactly the same as that of Lemma~\ref{Lemd4} with the only difference that $\mathfrak{b}_{m,1,\varepsilon}$ cannot be bounded from above by a constant which does not depend on $\varepsilon$ for all $m\in (-1,0]$, though it is finite due to \eqref{app5}.
\end{proof}

%%%%%%%%%%%%%%%%
%%%%%%%%%%%%%%%%
\subsection{Weighted $L^{q_1}$-estimate}\label{sec3.2}
%%%%%%%%%%%%%%%%
%%%%%%%%%%%%%%%%

The last estimate which does not depend on $\varepsilon\in (0,\varepsilon_\varrho)$ is the following weighted $L^{q_1}$-estimate, the exponent $q_1$ being defined in \eqref{p1}. 

%%%%%%%%%%%%%%%%
\begin{lemma}\label{Lemd5}
\refstepcounter{Num2Const}\label{kap8}
Consider $\varepsilon\in (0,\varepsilon_\varrho)$ and $f^{in}\in X_0^+\cap X_{1+\lambda}$ satisfying \eqref{d11} and let $g_\varepsilon=\Psi_\varepsilon(\cdot;f^{in})$ be given by \eqref{d2}. If $f^{in}$ also belongs to $L^{q_1}((0,\infty),y^{m_1}\mathrm{d}y)$, then there is $\kappa_{\ref{kap8}}>0$ such that
\begin{equation*}
\int_0^\infty x^{m_1} g_\varepsilon(s,x)^{q_1}\ \mathrm{d}x \le \max\left\{ \int_0^\infty x^{m_1} f^{in}(x)^{q_1}\ \mathrm{d}x , \kappa_{\ref{kap8}} \mathcal{M}_{\mu_1,\varepsilon}^{q_1} \right\}\ ,
\end{equation*} 
where $\mu_1 := (m_1+1+q_1(\lambda-2))/q_1>m_0$ and
\begin{equation*}
\mathcal{M}_{\mu_1,\varepsilon} := \sup_{s\ge 0}\{ M_{\mu_1}(g_\varepsilon(s))\}\ .
\end{equation*}
\end{lemma}
%%%%%%%%%%%%%%%%

\begin{proof}
We first observe that, as $\mu_1\in (m_0,\lambda)$ by \eqref{p1}, Lemma~\ref{Lemd3}, Lemma~\ref{Lemd5}, and H\"older's inequality imply that $\mathcal{M}_{\mu_1,\varepsilon}$ is finite. We next set 
\begin{equation*}
L_\varepsilon(s) := \frac{1}{q_1} \int_0^\infty y^{m_1} g_\varepsilon(s,y)^{q_1}\ \mathrm{d}y\ , \qquad s\ge 0\ , 
\end{equation*} 	
and infer from \eqref{d3} that
\begin{align}
\frac{\mathrm{d}}{\mathrm{d}s} L_\varepsilon(s) & = -(2q_1-m_1-1) L_\varepsilon(s) + \int_0^\infty x^{m_1} g_\varepsilon(s,x)^{q_1-1} \mathcal{C}g_\varepsilon(s,x)\ \mathrm{d}x \nonumber\\
& \qquad + \int_0^\infty x^{m_1} g_\varepsilon(s,x)^{q_1-1} \mathcal{F}_\varepsilon g_\varepsilon(s,x)\ \mathrm{d}x\ . \label{d12}
\end{align}
On the one hand, we use a monotonicity argument as in \cite{Buro83, LaurXX, LaMi02c, MiRR03} to estimate the contribution of the coagulation term. More precisely, thanks to the symmetry of $K$ and the subadditivity of $x\mapsto x^{m_1}$,
\begin{align*}
R_\varepsilon(s) & := \int_0^\infty x^{m_1} g_\varepsilon(s,x)^{q_1-1} \mathcal{C}g_\varepsilon(s,x)\ \mathrm{d}x \\
& = \frac{1}{2} \int_0^\infty \int_0^\infty (x+y)^{m_1} K(x,y) g_\varepsilon(s,x+y)^{q_1-1} g_\varepsilon(s,x) g_\varepsilon(s,y)\ \mathrm{d}y \mathrm{d}x \\
& \qquad - \int_0^\infty \int_0^\infty x^{m_1} K(x,y) g_\varepsilon(s,x)^{q_1} g_\varepsilon(s,y)\ \mathrm{d}y \mathrm{d}x \\
& \le \frac{1}{2} \int_0^\infty \int_0^\infty \left( x^{m_1} + y^{m_1} \right) K(x,y) g_\varepsilon(s,x+y)^{q_1-1} g_\varepsilon(s,x) g_\varepsilon(s,y)\ \mathrm{d}y \mathrm{d}x \\
& \qquad - \int_0^\infty \int_0^\infty x^{m_1} K(x,y) g_\varepsilon(s,x)^{q_1} g_\varepsilon(s,y)\ \mathrm{d}y \mathrm{d}x \\
& = \int_0^\infty \int_0^\infty x^{m_1} K(x,y) g_\varepsilon(s,x+y)^{q_1-1} g_\varepsilon(s,x) g_\varepsilon(s,y)\ \mathrm{d}y \mathrm{d}x \\
& \qquad - \int_0^\infty \int_0^\infty x^{m_1} K(x,y) g_\varepsilon(s,x)^{q_1} g_\varepsilon(s,y)\ \mathrm{d}y \mathrm{d}x\ .
\end{align*}
We now use Young's inequality to obtain
\begin{align*}
R_\varepsilon(s) & \le \int_0^\infty \int_0^\infty x^{m_1} K(x,y) \left[ \frac{q_1-1}{q_1} g_\varepsilon(s,x+y)^{q_1} + \frac{1}{q_1} g_\varepsilon(s,x)^{q_1} \right] g_\varepsilon(s,y)\ \mathrm{d}y \mathrm{d}x \\
& \qquad - \int_0^\infty \int_0^\infty x^{m_1} K(x,y) g_\varepsilon(s,x)^{q_1} g_\varepsilon(s,y)\ \mathrm{d}y \mathrm{d}x \\
& = \frac{q_1-1}{q_1} \int_0^\infty \int_0^\infty x^{m_1} K(x,y) g_\varepsilon(s,x+y)^{q_1} g_\varepsilon(s,y)\ \mathrm{d}y \mathrm{d}x \\
& \qquad - \frac{q_1-1}{q_1} \int_0^\infty \int_0^\infty x^{m_1} K(x,y) g_\varepsilon(s,x)^{q_1} g_\varepsilon(s,y)\ \mathrm{d}y \mathrm{d}x \\
& = \frac{q_1-1}{q_1} \int_0^\infty \int_y^\infty x^{m_1} K(x-y,y) g_\varepsilon(s,x)^{q_1} g_\varepsilon(s,y)\ \mathrm{d}x \mathrm{d}y \\
& \qquad - \frac{q_1-1}{q_1} \int_0^\infty \int_0^\infty x^{m_1} K(x,y) g_\varepsilon(s,x)^{q_1} g_\varepsilon(s,y)\ \mathrm{d}x \mathrm{d}y \ .
\end{align*}
Owing to the monotonicity of $x\mapsto x^{m_1} K(x,y)$ for all $y\in (0,\infty)$, the right hand side of the previous inequality is non-positive. Consequently,
\begin{equation}
R_\varepsilon(s) = \int_0^\infty x^{m_1} g_\varepsilon(s,x)^{q_1-1} \mathcal{C}g_\varepsilon(s,x)\ \mathrm{d}x \le 0\ . \label{d13}
\end{equation}
On the other hand, it follows from \eqref{aa2}, \eqref{be3}, and Fubini's theorem that 
\begin{align*}
S_\varepsilon(s) & := \int_0^\infty x^{m_1} g_\varepsilon(s,x)^{q_1-1} \int_x^\infty a(y) b_\varepsilon(x,y) g_\varepsilon(s,y)\ \mathrm{d}y \mathrm{d}x \\
& = a_0 \int_0^\infty y^{\lambda-2} g_\varepsilon(s,y) \int_0^{y} x^{m_1} B_\varepsilon\left( \frac{x}{y} \right) g_\varepsilon(s,x)^{q_1-1}\ \mathrm{d}x \mathrm{d}y \ .
\end{align*}
Since
\begin{align*}
& \int_0^{y} x^{m_1} B_\varepsilon\left( \frac{x}{y} \right) g_\varepsilon(s,x)^{q_1-1}\ \mathrm{d}x \\ 
& \qquad\qquad \le \left( \int_0^y x^{m_1} g_\varepsilon(s,x)^{q_1}\ \mathrm{d}x \right)^{(q_1-1)/q_1} \left( \int_0^y x^{m_1} B_\varepsilon\left( \frac{x}{y} \right)^{q_1}\ \mathrm{d}x \right)^{1/q_1} \\
& \qquad\qquad \le q_1^{(q_1-1)/q_1} \mathfrak{b}_{m_1,q_1,\varepsilon}^{1/q_1} L_\varepsilon(s)^{(q_1-1)/q_1} y^{(m_1+1)/q_1} \\
& \qquad\qquad \le q_1^{(q_1-1)/q_1} (1+\mathfrak{b}_{m_1,q_1})^{1/q_1} L_\varepsilon(s)^{(q_1-1)/q_1} y^{(m_1+1)/q_1}\ ,
\end{align*}
by \eqref{app4} and H\"older's inequality, we conclude that
\begin{align}
& \int_0^\infty x^{m_1} g_\varepsilon(s,x)^{q_1-1} \mathcal{F}_\varepsilon g_\varepsilon(s,x)\ \mathrm{d}x \le S_\varepsilon(s) \nonumber \\
& \qquad\qquad \le a_0 q_1^{(q_1-1)/q_1} (1+\mathfrak{b}_{m_1,q_1})^{1/q_1} M_{\mu_1}(g_\varepsilon(s)) L_\varepsilon(s)^{(q_1-1)/q_1}\ . \label{d14}
\end{align}
Collecting \eqref{d12}, \eqref{d13}, and \eqref{d14}, we end up with
\begin{align*}
\frac{\mathrm{d}}{\mathrm{d}s}L_\varepsilon(s) & \le -(2q_1-m_1-1) L_\varepsilon(s) + a_0 q_1^{(q_1-1)/q_1} (1+\mathfrak{b}_{m_1,q_1})^{1/q_1} \mathcal{M}_{\mu_1,\varepsilon} L_\varepsilon(s)^{(q_1-1)/q_1} \\
& = \frac{2q_1-m_1-1}{q_1^{1/q_1}} L_\varepsilon(s)^{(q_1-1)/q_1} \left[ \kappa_{\ref{kap8}}^{1/q_1} \mathcal{M}_{\mu_1,\varepsilon} - q_1^{1/q_1} L_\varepsilon(s)^{1/q_1} \right]
\end{align*}
with $\kappa_{\ref{kap8}} = (a_0 q_1)^{q_1} (1+\mathfrak{b}_{m_1,q_1}) / (2q_1-m_1-1)^{q_1}$. Lemma~\ref{Lemd5} follows from the above differential inequality by the comparison principle.
\end{proof}

%%%%%%%%%%%%%%%%
%%%%%%%%%%%%%%%%
\subsection{$W^{1,1}$-estimate}\label{sec3.3}
%%%%%%%%%%%%%%%%
%%%%%%%%%%%%%%%%

It turns out that the weighted $L^{q_1}$-estimate derived in Lemma~\ref{Lemd5}, though at the heart of the proof of Theorem~\ref{ThmP2}, is not sufficient to prove Proposition~\ref{PropApp}, and the final estimate needed for the proof of Proposition~\ref{PropApp} is the following $W^{1,1}$-estimate which depends strongly on $\varepsilon\in (0,\varepsilon_\varrho)$.

%%%%%%%%%%%%%%%%
\begin{lemma}\label{Lemd6}
\refstepcounter{Num2Const}\label{kap9}
Consider $\varepsilon\in (0,\varepsilon_\varrho)$ and $f^{in}\in X_0^+\cap X_{1+\lambda}$ satisfying \eqref{d11} and let $g_\varepsilon=\Psi_\varepsilon(\cdot;f^{in})$ be given by \eqref{d2}. Assume also that $f^{in}\in X_{\lambda-2}\cap W^{1,1}(0,\infty)$. Then there is $\kappa_{\ref{kap9}}(\varepsilon)>0$ depending on $\varepsilon$ such that
\begin{equation*}
\|\partial_x g_\varepsilon(s)\|_1 \le \max\left\{ \|\partial_x f^{in}\|_1 , \kappa_{\ref{kap9}}(\varepsilon) \mathcal{M}_{\lambda-2,\varepsilon} \right\}\ , \qquad s\ge 0\ ,
\end{equation*}
where
\begin{equation*}
\mathcal{M}_{\lambda-2,\varepsilon} := \sup_{s\ge 0}\{ M_{\lambda-2}(g_\varepsilon(s))\}\ .
\end{equation*}
\end{lemma}
%%%%%%%%%%%%%%%%

\begin{proof}
We first note that $\mathcal{M}_{\lambda-2,\varepsilon}$ is finite according to Lemma~\ref{Lemd100}, as $\lambda-2\in (-1,0)$ by \eqref{aa1}. Introducing $G_\varepsilon:=\partial_x g_\varepsilon$, $\Sigma_\varepsilon :=\mathrm{sign}(G_\varepsilon)$, and using that $K(x,0)=0$, it follows from \eqref{d3a} that $G_\varepsilon$ solves
\begin{align}
\partial_s G_\varepsilon(s,x) & = - x \partial_x G_\varepsilon(s,x) - \left( 3+a(x)+ \int_0^\infty K(x,y) g_\varepsilon(s,y)\ \mathrm{d}y \right) G_\varepsilon(s,x) \nonumber \\
& \qquad + \frac{1}{2} \int_0^x K(y,x-y) g_\varepsilon(s,y) G_\varepsilon(s,x-y)\ \mathrm{d}y \nonumber \\
& \qquad + \frac{1}{2} \int_0^x \partial_1 K(y,x-y) g_\varepsilon(s,y) g_\varepsilon(s,x-y)\ \mathrm{d}y \nonumber \\
& \qquad - \left( \frac{\mathrm{d}a}{\mathrm{d}x}(x) + a(x) b_\varepsilon(x,x) + \int_0^\infty \partial_1 K(x,y) g_\varepsilon(s,y)\ \mathrm{d}y \right) g_\varepsilon(s,x) \nonumber \\
& \qquad + \int_x^\infty a(y) \partial_1 b_\varepsilon(x,y) g_\varepsilon(s,y)\ \mathrm{d}y \label{w1}
\end{align}	
for $(s,x)\in (0,\infty)^2$, where $\partial_1 K$ and $\partial_1 b_\varepsilon$ denote the partial derivatives with respect to the first variable of $K$ and $b_\varepsilon$, respectively. 

Let $s\ge 0$. We multiply \eqref{w1} by $\Sigma_\varepsilon$, integrate with respect to $x$ over $(0,\infty)$ and then infer from \eqref{aa2}, \eqref{be3}, and Fubini's theorem that
\begin{align*}
\frac{\mathrm{d}}{\mathrm{d}s} \|G_\varepsilon(s)\|_1 & \le - 2 \|G_\varepsilon(s)\|_1 - a_0 M_{\lambda-1}(|G_\varepsilon(s)|) \\
& \qquad - \int_0^\infty \int_0^\infty K(x,y) g_\varepsilon(s,y) |G_\varepsilon(s,x)|\ \mathrm{d}y\mathrm{d}x \\
& \qquad + \frac{1}{2} \int_0^\infty \int_0^\infty K(x,y) g_\varepsilon(s,y) |G_\varepsilon(s,x)|\ \mathrm{d}y\mathrm{d}x \\
& \qquad + \frac{3}{2} \int_0^\infty \int_0^\infty |\partial_1 K(x,y)| g_\varepsilon(s,y) g_\varepsilon(s,x)\ \mathrm{d}y\mathrm{d}x \\
& \qquad + a_0 \left( \lambda-1+B_\varepsilon(1) + \int_0^1 \left| \frac{\mathrm{d}B_\varepsilon}{\mathrm{d}z}(z) \right|\ \mathrm{d}z \right) M_{\lambda-2}(g_\varepsilon(s)) \ .
\end{align*}
Setting 
\begin{equation*}
\bar{B}_\varepsilon := 1 + B_\varepsilon(1) + \int_0^1 \left| \frac{\mathrm{d}B_\varepsilon}{\mathrm{d}z}(z) \right|\ \mathrm{d}z\ ,
\end{equation*}
which is finite according to \eqref{app5}, and observing that
\begin{equation*}
0 \le \partial_1 K(x,y) \le K_0 \left[ x^{\alpha-1} y^{\lambda-\alpha} + x^{\alpha} y^{\lambda-\alpha-1} \right]\ , \qquad (x,y)\in (0,\infty)^2\ ,
\end{equation*} 
due to \eqref{aa1} and \eqref{aa3}, we end up with
\begin{align*}
\frac{\mathrm{d}}{\mathrm{d}s} \|G_\varepsilon(s)\|_1 & \le - 2 \|G_\varepsilon(s)\|_1 + a_0 \bar{B}_\varepsilon M_{\lambda-2}(g_\varepsilon(s)) \\
& \qquad + \frac{3}{2}\int_0^\infty \int_0^\infty |\partial_1 K(x,y)| g_\varepsilon(s,y) g_\varepsilon(s,x)\ \mathrm{d}y\mathrm{d}x \\
& \le - 2 \|G_\varepsilon(s)\|_1 + a_0 \bar{B}_\varepsilon M_{\lambda-2}(g_\varepsilon(s)) \\
& \qquad + \frac{3K_0}{2} \left[ M_\alpha(g_\varepsilon(s)) M_{\lambda-\alpha-1}(g_\varepsilon(s)) + M_{\alpha-1}(g_\varepsilon(s)) M_{\lambda-\alpha}(g_\varepsilon(s)) \right]\ .
\end{align*}
We next infer from \eqref{aa1} and H\"older's inequality that
\begin{align*}
M_\alpha(g_\varepsilon(s)) & \le M_1(g_\varepsilon(s))^{(\alpha+2-\lambda)/(3-\lambda)} M_{\lambda-2}(g_\varepsilon(s))^{(1-\alpha)/(3-\lambda)} \ , \\
M_{\lambda-\alpha-1}(g_\varepsilon(s)) & \le M_1(g_\varepsilon(s))^{(1-\alpha)/(3-\lambda)} M_{\lambda-2}(g_\varepsilon(s))^{(\alpha+2-\lambda)/(3-\lambda)} \ , \\
M_{\alpha-1}(g_\varepsilon(s)) & \le M_1(g_\varepsilon(s))^{(\alpha+1-\lambda)/(3-\lambda)} M_{\lambda-2}(g_\varepsilon(s))^{(2-\alpha)/(3-\lambda)} \ , \\
M_{\lambda-\alpha}(g_\varepsilon(s)) & \le M_1(g_\varepsilon(s))^{(2-\alpha)/(3-\lambda)} M_{\lambda-2}(g_\varepsilon(s))^{(\alpha+1-\alpha)/(3-\lambda)} \ ,
\end{align*}
so that, by \eqref{d9} and \eqref{d11},
\begin{equation*}
M_\alpha(g_\varepsilon(s)) M_{\lambda-\alpha-1}(g_\varepsilon(s)) + M_{\alpha-1}(g_\varepsilon(s)) M_{\lambda-\alpha}(g_\varepsilon(s)) \le 2\varrho M_{\lambda-2}(g_\varepsilon(s))\ .
\end{equation*}
Collecting the above inequalities and using \eqref{app6}, we conclude that
\begin{equation*}
\frac{\mathrm{d}}{\mathrm{d}s} \|G_\varepsilon(s)\|_1 + 2 \|G_\varepsilon(s)\|_1 \le 2\kappa_{\ref{kap9}}(\varepsilon) \mathcal{M}_{\lambda-2,\varepsilon}\ ,
\end{equation*}
with $\kappa_{\ref{kap9}}(\varepsilon) := \left( a_0 \bar{B}_\varepsilon + 3\varrho K_0 \right)/2$. Integrating the previous differential inequality gives Lemma~\ref{Lemd6}.
\end{proof}

%%%%%%%%%%%%%%%%
%%%%%%%%%%%%%%%%
\subsection{Invariant Set}\label{sec3.4}
%%%%%%%%%%%%%%%%
%%%%%%%%%%%%%%%%

The analysis performed in the previous three sections now allows us to construct a compact and convex subset of $X_1$ which is left invariant by \eqref{d3}. Let us first recall that, owing to \eqref{p1}, the parameter $\mu_1$ (defined in Lemma~\ref{Lemd5}) satisfies
\begin{equation}
1+\lambda >\mu_1 = \frac{m_1+1+q_1(\lambda-2)}{q_1}>m_0>-\nu-1\ . \label{d16}
\end{equation}

For $\varepsilon\in (0,\varepsilon_\varrho)$, we define the subset $\mathcal{Z}_\varepsilon$ of $X_1^+$ as follows: $h\in\mathcal{Z}_\varepsilon$ if and only if $h$ satisfies the following conditions: 
\begin{subequations}\label{d15}
\begin{align}
& h\in X_1^+ \cap \bigcap_{m\ge \lambda-2} X_m \cap W^{1,1}(0,\infty)\ , \qquad M_1(h) = \varrho\ , \label{d15a}\\
& \int_0^\infty x \ln{(x)}\; h(x)\ \mathrm{d}x + \frac{3}{e(1-m_1)} M_{m_1}(h) \le \kappa_{\ref{kap1}}(m_1)\ , \label{d15b}\\
& M_m(h) \le \kappa_{\ref{kap4}}(m)\ , \qquad m\ge 1+\lambda\ , \label{d15c}\\
& M_{m_0}(h)\le \kappa_{\ref{kap7}}(m_0) \kappa_{\ref{kap4}}(1+\lambda)\ , \label{d15d} \\
& M_{\mu_1}(h) \le \kappa_{\ref{kap7}}(m_0)^{(1+\lambda-\mu_1)/(1+\lambda-m_0)} \kappa_{\ref{kap4}}(1+\lambda)\ , \label{d15e} \\
& \int_0^\infty x^{m_1} h(x)^{q_1}\ \mathrm{d}x \le \kappa_{\ref{kap8}} \kappa_{\ref{kap7}}(m_0)^{q_1(1+\lambda-\mu_1)/(1+\lambda-m_0)} \kappa_{\ref{kap4}}(1+\lambda)^{q_1}\ , \label{d15f} \\
& M_{\lambda-2}(h) \le \kappa_{\ref{kap100}}(\lambda-2,\varepsilon) \kappa_{\ref{kap4}}(1+\lambda)\ , \label{d15g} \\
& \|\partial_x h\|_1 \le \kappa_{\ref{kap9}}(\varepsilon) \kappa_{\ref{kap100}}(\lambda-2,\varepsilon) \kappa_{\ref{kap4}}(1+\lambda)\ . \label{d15h}
\end{align}
\end{subequations}
Note that we may assume that $E_\varrho: x\mapsto \varrho e^{-x}$ belongs to $\mathcal{Z}_\varepsilon$, after possibly taking larger constants in \eqref{d15} without changing their dependence with respect to the involved parameters. In particular, $\mathcal{Z}_\varepsilon$ is non-empty.

As we shall see now, the outcome of the analysis performed in the previous sections provides the invariance of $\mathcal{Z}_\varepsilon$ for the dynamics of \eqref{d3} when $\varepsilon\in (0,\varepsilon_\varrho)$. 

%%%%%%%%%%%%%%%%
\begin{lemma}\label{Lemd7}
Consider $\varepsilon\in (0,\varepsilon_\varrho)$ and $f^{in}\in \mathcal{Z}_\varepsilon$. Then $\Psi_\varepsilon(s;f^{in})\in \mathcal{Z}_\varepsilon$ for all $s\ge 0$. Furthermore, $\mathcal{Z}_\varepsilon$ is a non-empty, convex, and compact subset of $X_1$.
\end{lemma}
%%%%%%%%%%%%%%%%

\begin{proof}
Let $f^{in}\in \mathcal{Z}_\varepsilon$. Setting $g_\varepsilon=\Psi_\varepsilon(\cdot;f^{in})$, see \eqref{d2}, it satisfies \eqref{d9} by Lemma~\ref{Lemd1}, from which we readily obtain that $g_\varepsilon(s)\in X_1^+$ and $M_1(g_\varepsilon(s))=\varrho$ for all $s\ge 0$. 

Next, let $s\ge 0$. We infer from \eqref{d15b} and Lemma~\ref{Lemd1} (with $m=m_1$) that $g_\varepsilon(s)$ satisfies \eqref{d15b}. Also, since $f^{in}$ satisfies \eqref{d11} according to \eqref{d15b}, we are in a position to apply Lemma~\ref{Lemd3} for $m\ge 1+\lambda>1+\lambda-\alpha$ and deduce from \eqref{d15c} for $f^{in}$ that \eqref{d15c} is satisfied by $g_\varepsilon(s)$ for any $m\ge 1+\lambda$. This property (with $m=1+\lambda$) along with Lemma~\ref{Lemd4} (with $m=m_0$) guarantees that $g_\varepsilon(s)$ satisfies \eqref{d15d}. We further use \eqref{d15c} (with $m=1+\lambda$) and \eqref{d15d} that we just established for $g_\varepsilon$ together with \eqref{d16} and H\"older's inequality to obtain
\begin{align*}
M_{\mu_1}(g_\varepsilon(s)) & \le M_{1+\lambda}(g_\varepsilon(s))^{(\mu_1-m_0)/(1+\lambda-m_0)} M_{m_0}(g_\varepsilon(s))^{(1+\lambda-\mu_1)/(1+\lambda-m_0)} \\
& \le \kappa_{\ref{kap4}}(1+\lambda)^{(\mu_1-m_0)/(1+\lambda-m_0)} \left[ \kappa_{\ref{kap4}}(1+\lambda) \kappa_{\ref{kap7}}(m_0) \right]^{(1+\lambda-\mu_1)/(1+\lambda-m_0)} \\
& \le \kappa_{\ref{kap7}}(m_0)^{(1+\lambda-\mu_1)/(1+\lambda-m_0)} \kappa_{\ref{kap4}}(1+\lambda)\ .
\end{align*}
Hence, $g_\varepsilon(s)$ satisfies \eqref{d15e} for $s\ge 0$. We now combine the just established property \eqref{d15e} for $g_\varepsilon$ with Lemma~\ref{Lemd6} and realize that $g_\varepsilon(s)$ satisfies \eqref{d15f} for $s\ge 0$. Finally, since $f^{in}$ satisfies \eqref{d15g} and \eqref{d15h}, it follows at once from the already proved property \eqref{d15c} for $g_\varepsilon$ (for $m=1+\lambda$), Lemma~\ref{Lemd100}, and Lemma~\ref{Lemd6} that $g_\varepsilon(s)$ also satisfies \eqref{d15g} and \eqref{d15h}. Summarizing, we have shown that $g_\varepsilon(s)\in\mathcal{Z}_\varepsilon$ for all $s\ge 0$.

Next, the set $\mathcal{Z}_\varepsilon$ is convex and its compactness in $X_1$ follows from its boundedness in $X_{\lambda-2}\cap X_{1+\lambda}$, the compactness of the embedding of $W^{1,1}(1/R,R)$ in $L^1(1/R,R)$, which holds true for all $R>1$, and Vitali's theorem \cite[Theorem~2.24]{FoLe07}.
\end{proof}

To complete the proof of Proposition~\ref{PropApp}, the missing tile is the continuity of weak solutions to \eqref{d3} with respect to the initial condition which we establish now.

%%%%%%%%%%%%%%%%
\begin{lemma}\label{Lemd8}
Let $\varepsilon\in (0,\varepsilon_\varrho)$.
\begin{itemize}
	\item[(a)] For $s\ge 0$, the map $f^{in}\longmapsto\Psi_\varepsilon(s;f^{in})$, defined in \eqref{d2}, is continuous from $\mathcal{Z}_\varepsilon$ endowed with the norm topology of $X_1$ to itself. 
	\item[(b)] For $f^{in}\in \mathcal{Z}_\varepsilon$, the map $s\longmapsto \Psi_\varepsilon(s;f^{in})$ belongs to $C([0,\infty),X_1)$.
\end{itemize} 
In other words, $\Psi_\varepsilon: [0,\infty)\times \mathcal{Z}_\varepsilon \longrightarrow \mathcal{Z}_\varepsilon$ is a dynamical system for the norm topology of $X_1$.
\end{lemma}
%%%%%%%%%%%%%%%%

\begin{proof}[Proof of Lemma~\ref{Lemd8}~(a)]
Consider $(f_1^{in},f_2^{in})\in \mathcal{Z}_\varepsilon^2$ and put $g_{i,\varepsilon} := \Psi_\varepsilon(\cdot;f_i^{in})$, $i=1,2$.  Arguing as in the proof of \cite[Theorem~1.2~(c)]{LaurXX}, it follows from \eqref{d3} that, for $s\ge 0$,
\begin{align*}
& \frac{\mathrm{d}}{\mathrm{d}s} \int_0^\infty W(x) |g_{1,\varepsilon}(s,x)-g_{2,\varepsilon}(s,x)|\ \mathrm{d}x \\
& \qquad \le \int_0^\infty \left[ x \frac{\mathrm{d}W}{\mathrm{d}x}(x) -W(x) \right] |g_{1,\varepsilon}(s,x)-g_{2,\varepsilon}(s,x)|\ \mathrm{d}x \\
& \qquad\quad + \left[ 9K_0 v_\varepsilon(s) + a_0\mathfrak{b}_{\alpha,1,\varepsilon} \right] \int_0^\infty W(x) |g_{1,\varepsilon}(s,x)-g_{2,\varepsilon}(s,x)|\ \mathrm{d}x\ , 
\end{align*}
where $W(x) = x^\alpha+x^\lambda$, $x\ge 0$, and
\begin{equation*}
v_\varepsilon(s) := M_\alpha(g_{1,\varepsilon}(s)) + M_\alpha(g_{2,\varepsilon}(s)) + M_{2\lambda-\alpha}(g_{1,\varepsilon}(s)) + M_{2\lambda-\alpha}(g_{2,\varepsilon}(s))\ .
\end{equation*}
Since both $f_1^{in}$ and $f_2^{in}$ belong to $\mathcal{Z}_\varepsilon$, so do $g_{1,\varepsilon}(s)$ and $g_{2,\varepsilon}(s)$ for all $s\ge 0$ by Lemma~\ref{Lemd7}. Consequently, as $m_0<\alpha<2\lambda-\alpha\le 1+\lambda$ by \eqref{aa1} and \eqref{aa11}, 
\begin{equation*}
V_\varepsilon := \sup_{s\ge 0}\{v_\varepsilon(s)\} < \infty\ .
\end{equation*} 
In addition,
\begin{equation*}
x W'(x) - W(x) = (\alpha-1) x^\alpha + (\lambda-1) x^\lambda \le x^\lambda \le W(x)\ , \qquad x\in (0,\infty)\ ,
\end{equation*}
by \eqref{aa1} and we infer from \eqref{app4a} and the previous differential inequality that, for $s\ge 0$, 
\refstepcounter{Num2Const}\label{kap10}
\begin{equation}
\int_0^\infty W(x) |g_{1,\varepsilon}(s,x)-g_{2,\varepsilon}(s,x)|\ \mathrm{d}x \le e^{\kappa_{\ref{kap10}}(\varepsilon)s} \int_0^\infty W(x) |f_1^{in}(x)-f_2^{in}(x)|\ \mathrm{d}x\ , \label{d17}
\end{equation}
with $\kappa_{\ref{kap10}}(\varepsilon) := 1 + 9K_0 V_\varepsilon + a_0 \mathfrak{b}_{\alpha,1,\varepsilon}$.

\refstepcounter{Num2Const}\label{kap11}
Now, $W(x)\ge x$ for $x\ge 0$ as $\alpha\le 1 < \lambda$, while, for $R>1$, it follows from \eqref{aa1} and \eqref{aa11} that
\begin{align*}
\int_0^\infty W(x) |f_1^{in}(x)-f_2^{in}(x)|\ \mathrm{d}y & \le \int_0^{1/R} W(x) [f_1^{in}(x)+f_2^{in}(x)]\ \mathrm{d}x \\
& \qquad + \int_{1/R}^R W(x) |f_1^{in}(x)-f_2^{in}(x)|\ \mathrm{d}x \\
& \qquad + \int_R^\infty W(x) [f_1^{in}(x)+f_2^{in}(x)]\ \mathrm{d}x \\
& \le \left( R^{m_0-\alpha} + R^{m_0-\lambda} \right) \left[ M_{m_0}(f_1^{in}) + M_{m_0}(f_2^{in}) \right] \\
& \qquad + \left( R^{1-\alpha} + R^{\lambda-1} \right) \int_{1/R}^R x |f_1^{in}(x)-f_2^{in}(x)|\ \mathrm{d}x \\
& \qquad + \left( R^{\alpha-1-\lambda} + R^{-1} \right) \left[ M_{1+\lambda}(f_1^{in}) + M_{1+\lambda}(f_2^{in}) \right] \\
& \le \kappa_{\ref{kap11}} \left[ R^{m_0-\alpha} + R^{-1} + R^{1-\alpha} \int_0^\infty x |f_1^{in}(x)-f_2^{in}(x)|\ \mathrm{d}x \right]\ ,
\end{align*}
the last inequality relying on the property $f_i^{in}\in \mathcal{Z}_\varepsilon$, $i=1,2$. Combining \eqref{d17} and the previous inequalities gives, for $s\ge 0$,
\begin{align*}
& \int_0^\infty x |g_{1,\varepsilon}(s,x)-g_{2,\varepsilon}(s,x)|\ \mathrm{d}x \\
& \qquad\qquad \le \kappa_{\ref{kap11}} e^{\kappa_{\ref{kap10}}(\varepsilon)s} \omega\left( \int_0^\infty x |f_1^{in}(x)-f_2^{in}(x)|\ \mathrm{d}x \right)\ ,
\end{align*}
with
\begin{equation*}
\omega(r) := \inf_{R>1}\left\{ R^{m_0-\alpha} + R^{-1} + R^{1-\alpha} r \right\}\ , \qquad r>0\ .
\end{equation*}
Since $\omega(r)\longrightarrow 0$ as $r\to 0$, the claimed continuity follows.
\end{proof}

\begin{proof}[Proof of Lemma~\ref{Lemd8}~(b)]
Set $g_\varepsilon=\Psi_\varepsilon(\cdot;f^{in})$. Let $s\ge 0$. We infer from \eqref{aa1}, \eqref{aa2}, \eqref{aa3}, \eqref{aa11}, \eqref{be3}, \eqref{app6}, \eqref{d3a}, \eqref{d9}, \eqref{d11},  and H\"older's inequality that
\begin{align*}
\int_0^\infty \frac{|\partial_s g_\varepsilon(s,x)|}{1+x}\ \mathrm{d}x & \le \|\partial_x g_\varepsilon(s)\|_1 + 2 M_0(g_\varepsilon(s)) + 3 K_0 M_\alpha(g_\varepsilon(s)) M_{\lambda-\alpha}(g_\varepsilon(s)) \\
& \qquad + a_0 \left( 1 + \mathfrak{b}_{0,1,\varepsilon} \right) M_{\lambda-1}(g_\varepsilon(s)) \\
& \le \|\partial_x g_\varepsilon(s)\|_1 + 2 \varrho^{(2-\lambda)/(3-\lambda)} M_{\lambda-2}(g_\varepsilon(s))^{1/(3-\lambda)} \\
& \qquad + 3 K_0 \varrho^{(\lambda-2m_0)/(1-m_0)} M_{m_0}(g_\varepsilon(s))^{(2-\lambda)/(1-m_0)} \\
& \qquad + a_0 \left( 1 + \mathfrak{b}_{0,1,\varepsilon} \right) \varrho^{1/(3-\lambda)} M_{\lambda-2}(g_\varepsilon(s))^{(2-\lambda)/(3-\lambda)} \ .
\end{align*}
\refstepcounter{Num2Const}\label{kap12}
Since $g_\varepsilon(s)\in \mathcal{Z}_\varepsilon$ by Lemma~\ref{Lemd7}, we further obtain
\begin{equation*}
\int_0^\infty \frac{|\partial_s g_\varepsilon(s,x)|}{1+x}\ \mathrm{d}x \le \kappa_{\ref{kap12}}(\varepsilon)\ , \qquad s\ge 0\ . 
\end{equation*}
Hence, for $s_2> s_1\ge 0$ and $R\ge 1$,
\begin{align*}
\int_0^\infty x |g_\varepsilon(s_2,x)-g_\varepsilon(s_1,x)|\ \mathrm{d}y & \le R(1+R) \int_0^R  \frac{|g_\varepsilon(s_2,x)-g_\varepsilon(s_1,x)|}{1+x}\ \mathrm{d}x \\
& \qquad + R^{-\lambda} \int_R^\infty x^{1+\lambda} \left[ g_\varepsilon(s_2,x) + g_\varepsilon(s_1,x) \right]\ \mathrm{d}x \\
& \le 2 R^2 \int_{s_1}^{s_2} \int_0^\infty  \frac{|\partial_s g_\varepsilon(s,x)|}{1+x}\ \mathrm{d}x \mathrm{d}s + 2 R^{-\lambda} \sup_{s\ge 0}\{ M_{1+\lambda}(g_\varepsilon(s))\} \\
& \le 2 R^2 \kappa_{\ref{kap12}}(\varepsilon) (s_2-s_1) + 2 R^{-\lambda} \kappa_{\ref{kap4}}(1+\lambda)\ .
\end{align*}
Choosing $R=(s_2-s_1)^{-1/(\lambda+2)}$ if $s_2-s_1<1$ and $R=1$ otherwise in the previous inequality, we are led to  
\begin{align*}
\int_0^\infty x |g_\varepsilon(s_2,x)-g_\varepsilon(s_1,x)|\ \mathrm{d}x \le 2 \left[ \kappa_{\ref{kap12}}(\varepsilon) + \kappa_{\ref{kap4}}(1+\lambda) \right] \left( (s_2-s_1)^{\lambda/(\lambda+2)} +s_2-s_1 \right)\ ,
\end{align*}
which provides the claimed continuity.
\end{proof}

We have now established all the properties required to prove Proposition~\ref{PropApp}.

\begin{proof}[Proof of Proposition~\ref{PropApp}]
Let $\varepsilon\in (0,\varepsilon_\varrho)$. Owing to Lemma~\ref{Lemd7} and Lemma~\ref{Lemd8}, $\Psi_\varepsilon$ is a dynamical system on $\mathcal{Z}_\varepsilon$ endowed with the norm topology of $X_1$ and $\mathcal{Z}_\varepsilon$ is a non-empty, convex, and compact subset of $X_1$, which is additionally left positively invariant by $\Psi_\varepsilon$. A consequence of Schauder's fixed point theorem, see \cite[Proposition~22.13]{Aman90} or \cite[Proof of Theorem~5.2]{GPV04}, implies that there is $\varphi_\varepsilon\in\mathcal{Z}_\varepsilon$ such that $\Psi_\varepsilon(s;\varphi_\varepsilon)=\varphi_\varepsilon$ for all $s\ge 0$. In other words, $\varphi_\varepsilon$ is a stationary solution to \eqref{d3a}, from which we deduce that it satisfies \eqref{app8}.  Also, since $\varphi_\varepsilon$ lies in $\mathcal{Z}_\varepsilon$, it has the properties \eqref{x0} due to \eqref{d15b}, \eqref{d15c}, \eqref{d15d}, and \eqref{d15f}. 
\end{proof}

%%%%%%%%%%%%%%%%
%%%%%%%%%%%%%%%%
\section{Self-similar solutions} \label{sec4}
%%%%%%%%%%%%%%%%
%%%%%%%%%%%%%%%%

In this section, we assume that $K$, $a$, and $b$ are coagulation and fragmentation coefficients satisfying \eqref{aaCFC} and we fix $\varrho\in (0,\varrho_\star)$. For $\varepsilon\in (0,\varepsilon_\varrho)$, it follows from Proposition~\ref{PropApp} that there is 
\begin{equation*}
\varphi_\varepsilon \in X_1^+ \cap L^{q_1}((0,\infty),x^{m_1}\mathrm{d}x) \cap W^{1,1}(0,\infty) \cap \bigcap_{m\ge \lambda - 2} X_m
\end{equation*}
satisfying \eqref{app8}, 
\begin{align}
M_1(\varphi_\varepsilon) & = \varrho\ , \label{4.1} \\
\sup_{\varepsilon\in (0,\varepsilon_\varrho)}\left\{ M_{m_0}(\varphi_\varepsilon) \right\} & + \sup_{\varepsilon\in (0,\varepsilon_\varrho)} \left\{ \int_0^\infty x^{m_1} \varphi_\varepsilon(x)^{q_1}\ \mathrm{d}x \right\} < \infty\ , \label{4.2}
\end{align}
and
\begin{equation}
\sup_{\varepsilon\in (0,\varepsilon_\varrho)}\left\{ M_m(\varphi_\varepsilon) \right\} < \infty \label{4.3}
\end{equation}
for all $m\ge 1+\lambda$. Since $q_1>1$ and $m_1<1$, we infer from \eqref{4.1}, \eqref{4.2}, the reflexivity of $L^{q_1}((0,\infty),x^{m_1}\mathrm{d}x)$, and Dunford-Pettis' theorem that there are $\varphi\in X_{m_1}\cap L^{q_1}((0,\infty),x^{m_1}\mathrm{d}x)$ and a subsequence $(\varphi_{\varepsilon_n})_{n\ge 1}$ of $(\varphi_\varepsilon)_{\varepsilon\in (0,\varepsilon_\varrho)}$ such that
\begin{equation}
\varphi_{\varepsilon_n} \rightharpoonup \varphi \;\text{ in }\; X_{m_1} \;\text{ and in }\; L^{q_1}((0,\infty),x^{m_1}\mathrm{d}x)\ . \label{4.5}
\end{equation} 
Combining \eqref{4.2}, \eqref{4.3}, and \eqref{4.5}, we further obtain that $\varphi\in X_{m_0}$ and
\begin{equation}
\varphi\in X_m \;\text{ and }\; \varphi_{\varepsilon_n} \rightharpoonup \varphi \;\text{ in }\; X_m\ , \qquad m>m_0\ . \label{4.6}
\end{equation}
Since the positive cone $X_1^+$ of $X_1$ is weakly closed in $X_1$, we infer from \eqref{4.1} and \eqref{4.6} (with $m=1$) that
\begin{equation}
\varphi\in X_1^+ \;\text{ and }\; M_1(\varphi) = \varrho\ . \label{4.7}
\end{equation} 
We are left with taking the limit $\varepsilon\to 0$ in \eqref{app8}. To this end, consider $\vartheta\in \Theta_1$, the space $\Theta_1$ being defined in \eqref{a6c}, and note that
\begin{equation}
|\vartheta(x)| \le \|\partial_x\vartheta\|_\infty x\ , \qquad x\in [0,\infty)\ . \label{4.75}
\end{equation}
Then $x\mapsto \vartheta(x)/x$ belongs to $L^\infty(0,\infty)$ and it readily follows from \eqref{4.6} (with $m=1$) that
\begin{align}
\lim_{n\to\infty} \int_0^\infty [ \vartheta(x) - x \partial_x \vartheta(x) ] \varphi_{\varepsilon_n}(x)\ \mathrm{d}x & = \lim_{n\to\infty} \int_0^\infty x \left[ \frac{\vartheta(x)}{x} - \partial_x\vartheta(x) \right] \varphi_{\varepsilon_n}(x)\ \mathrm{d}x \nonumber \\
& = \int_0^\infty [ \vartheta(x) - x \partial_x \vartheta(x) ] \varphi(x)\ \mathrm{d}x\ . \label{4.8}
\end{align}
Similarly, $\chi_\vartheta\in L^\infty((0,\infty)^2)$ and we argue as in \cite{Stew89}, see also \cite{BLLxx}, to deduce from \eqref{aa1}, \eqref{aa3}, \eqref{aa11}, and \eqref{4.6} (with $m=\alpha$ and $m=\lambda-\alpha$) that
\begin{align}
& \lim_{n\to\infty} \int_0^\infty \int_0^\infty K(x,y) \chi_\vartheta(x,y) \varphi_{\varepsilon_n}(x) \varphi_{\varepsilon_n}(y)\ \mathrm{d}y \mathrm{d}x \nonumber \\
& \qquad\qquad = \int_0^\infty \int_0^\infty K(x,y) \chi_\vartheta(x,y) \varphi(x) \varphi(y)\ \mathrm{d}y \mathrm{d}x\ . \label{4.9}
\end{align}
Finally, by \eqref{aa1}, \eqref{aa2}, and \eqref{4.6} (with $m=\lambda$), 
\begin{equation}
\big[ y \mapsto y a(y) \varphi_{\varepsilon_n}(y) \big] \rightharpoonup \big[ y \mapsto y a(y) \varphi(y) \big] \;\text{ in }\; L^1(0,\infty)\ , \label{4.10}
\end{equation}
while \eqref{be}, \eqref{app0}, and \eqref{4.75} entail, for $y\in (0,\infty)$,
\begin{align}
\left| \frac{N_{\vartheta,\varepsilon_n}(y)}{y} \right| & \le \frac{|\vartheta(y)|}{y} + \frac{1}{y} \int_0^1 |\vartheta(yz)| B_{\varepsilon_n}(z)\ \mathrm{d}z \nonumber \\
& \le \|\partial_x\vartheta\|_\infty \left( 1 + \int_0^1 z B_{\varepsilon_n}(z)\ \mathrm{d}z \right) = 2  \|\partial_x\vartheta\|_\infty\ . \label{4.11}
\end{align}
Using once more \eqref{4.75}, we obtain, for $y\in (0,\infty)$, 
\begin{align*}
\left| \int_0^y \vartheta(x) b_{\varepsilon_n}(x,y)\ \mathrm{d}x - \int_0^y \vartheta(x) b(x,y)\ \mathrm{d}x \right| & = \left| \int_0^1 \vartheta(yz) [B_{\varepsilon_n}(z)-B(z)]\ \mathrm{d}z \right| \\
& \le y \|\partial_x\vartheta\|_\infty \int_0^1 z |B_{\varepsilon_n}(z)-B(z)|\ \mathrm{d}z\ .
\end{align*}
Hence, thanks to \eqref{app2} (with $(m,p)=(1,1)$),
\begin{equation*}
\lim_{n\to\infty} \frac{1}{y} \int_0^y \vartheta(x) b_{\varepsilon_n}(x,y)\ \mathrm{d}x = \frac{1}{y} \int_0^y \vartheta(x) b(x,y)\ \mathrm{d}x\ ,
\end{equation*}
which implies, in turn,
\begin{equation}
\lim_{n\to\infty} \frac{N_{\vartheta,\varepsilon_n}(y)}{y} = \frac{N_{\vartheta}(y)}{y} \ , \qquad y\in (0,\infty)\ . \label{4.12}
\end{equation}
Due to \eqref{4.10}, \eqref{4.11}, and \eqref{4.12}, we are in a position to apply \cite[Proposition~2.61]{FoLe07} (which is a consequence of Dunford-Pettis' and Egorov's theorems) and conclude that
\begin{equation}
\lim_{n\to\infty} \int_0^\infty a(y) N_{\vartheta,\varepsilon_n}(y) \varphi_{\varepsilon_n}(y)\ \mathrm{d}y = \int_0^\infty a(y) N_{\vartheta}(y) \varphi(y)\ \mathrm{d}y\ . \label{4.13}
\end{equation}
Having established \eqref{4.8}, \eqref{4.9}, and \eqref{4.13}, we may take the limit $\varepsilon\to 0$ in \eqref{app8} and deduce that $\varphi$ satisfies \eqref{x8}, thereby completing the proof of Theorem~\ref{ThmP2}.

%%%%%%%%%%%%%%%%
%%%%%%%%%%%%%%%%
\bibliographystyle{siam}
\bibliography{SelfSimilarCF}
%%%%%%%%%%%%%%%%
%%%%%%%%%%%%%%%%

\end{document}